\documentclass[12pt]{amsart}

\usepackage[utf8]{inputenc}
\usepackage{amssymb,amsfonts,amsmath,amsthm,mathrsfs,dsfont,microtype,mathtools}
\usepackage[dvipsnames]{xcolor}
\usepackage[shortlabels]{enumitem}
\usepackage{tikz}
\usetikzlibrary{arrows.meta,backgrounds}

\usepackage{cite}

\usepackage{hyperref}
\hypersetup{
	colorlinks=true,
	linkcolor=blue,
	citecolor=blue,
	urlcolor=blue,
	pdfauthor={
		Fernando de Ávila Silva,
		Alexandre Kirilov,
		André Pedroso Kowacs
	},
	pdftitle={
		Sobolev Regularity in Mixed and Isotropic Scales
		for Vector Fields on the Torus
	}
}

 \usepackage[a4paper,margin=2.54cm]{geometry}

\usepackage{setspace}
\numberwithin{equation}{section}
\allowdisplaybreaks[2]
\theoremstyle{plain}
\newtheorem{theorem}{Theorem}[section]
\newtheorem{proposition}[theorem]{Proposition}
\newtheorem{lemma}[theorem]{Lemma}
\newtheorem{corollary}[theorem]{Corollary}

\theoremstyle{definition}
\newtheorem{definition}[theorem]{Definition}
\newtheorem{example}[theorem]{Example}

\theoremstyle{remark}
\newtheorem{remark}[theorem]{Remark}

\newcommand{\R}{\mathbb{R}}

\newcommand{\T}{\mathbb{T}}
\newcommand{\Hmix}[1]{H^{#1}_{\operatorname{mix}}(\T^2)}

\title[Sobolev Regularity in Mixed and Isotropic Scales]
{Sobolev Regularity in Mixed and Isotropic Scales \\ 
	for Vector Fields on the Torus}
\author{Fernando de {\'A}vila Silva}
\address{
	Universidade Federal do Paran\'a,
	Departamento de Matem\'atica,
	CEP 81531-980, Curitiba, Brazil
}
\email{fernando.avila@ufpr.br}

\author{Alexandre Kirilov}
\address{
	Universidade Federal do Paran\'a,
	Departamento de Matem\'atica,
	CEP 81531-980, Curitiba, Brazil
}
\email{akirilov@ufpr.br}

\author{Andr\'e Pedroso Kowacs}
\address{
	Departamento de Matem\'atica,
	Instituto de Ci\^encias Matem\'aticas e de Computa\c{c}\~ao,
	Universidade de S\~ao Paulo, S\~ao Paulo, Brazil
}
\email{andrekowacs@gmail.com}



\date{}

\begin{document}
	
	\begin{abstract}
		We study regularity, existence, and uniqueness of solutions to
		vector fields on the two-dimensional torus in Sobolev
		spaces of dominating mixed smoothness, which measure regularity
		separately in the two variables.
		
		For constant-coefficient vector fields, we obtain families of mixed 
		smoothness estimates describing how the gain or loss of regularity 
		can be distributed between the two variables. In the nonreal case,
		a gain of one derivative can be distributed between the two directions, 
		whereas for real irrational coefficients the loss is governed by the 
		irrationality measure of the coefficient. We also establish sharpness 
		below the corresponding arithmetic threshold and describe the 
		rational and Liouville obstructions.
		
		For real-valued variable coefficients, direct estimates for the
		periodic Fourier-mode equations yield mixed smoothness regularity
		results and their isotropic and classical consequences. We then use
		a periodic conjugation to the averaged constant-coefficient normal
		form. Although this conjugation introduces an additional loss in the
		mixed smoothness scale, it preserves isotropic Sobolev orders and
		therefore transfers the sharp constant-coefficient isotropic theory
		to the variable-coefficient setting. In the nonresonant regimes, we
		also obtain existence and uniqueness under the natural zero-mean
		compatibility condition.
	\end{abstract}
	
	\maketitle
	
		{
		\small  
		\setstretch{1} 
		\tableofcontents 
	}

%======================================
%======================================
\section{Introduction}
%======================================
%======================================

Regularity, existence, and uniqueness of solutions to periodic
first-order equations are strongly influenced by the arithmetic
properties of their coefficients. For vector fields and evolution
operators on tori, these questions are closely related to the
classical theories of global hypoellipticity and global solvability
and have been studied from several perspectives; see, for instance,
\cite{GW1972_pams,GW1973_top,GW1973_tams,CC2000_cpde,
	Petr2011_tams,Berg1999_tams,
	KMR2021_jfa,Kiri_Kow_Wagn_ToroEsfera}.

Beyond the qualitative smoothness theory, one may ask how much
Sobolev regularity is gained or lost when inverting such operators.
A quantitative isotropic Sobolev theory was developed in
\cite{KowKir2026-jfa}, where quantitative loss-of-derivative
estimates were obtained for Fourier multipliers on the torus. In
particular, for a constant real vector field
\[
L=\partial_t-a\partial_x,
\qquad
a\in\mathbb R,
\]
the isotropic Sobolev loss in the irrational case is governed by the
irrationality measure \(\mu(a)\), introduced in
Definition~\ref{def_complete_irrationality}. The isotropic scale, however,
measures the two Fourier variables jointly and does not distinguish
how regularity is distributed between the \(t\)- and \(x\)-directions.

The purpose of the present paper is to refine this analysis by
studying regularity separately in the two variables. We work in
Sobolev spaces of dominating mixed smoothness and consider both
constant-coefficient vector fields and real-valued
variable-coefficient vector fields on \(\mathbb T^2\). This leads to
families of directional regularity estimates that are invisible in a
single isotropic Sobolev exponent. For variable coefficients, we also
compare the estimates obtained by a direct Fourier-mode analysis with
those derived through a periodic conjugation to the averaged
constant-coefficient normal form.

More precisely, for \(s_1,s_2\in\mathbb R\), we consider
\[
H_{\operatorname{mix}}^{(s_1,s_2)}(\mathbb T^2)
=
\bigg\{
u\in\mathcal D'(\mathbb T^2):
\sum_{(\tau,\xi)\in\mathbb Z^2}
\langle\tau\rangle^{2s_1}
\langle\xi\rangle^{2s_2}
|\widehat u(\tau,\xi)|^2<\infty
\bigg\}.
\]

These spaces are classical in the theory of function spaces with
dominating mixed smoothness; see, for instance,
\cite{Tem1993_book,Tem2018_book,DungTemlyakovUllrich2018,
	KuhnSickelUllrich2015,CobosKuhnSickel2016,
	KuhnSickelUllrich2021,Triebel2019}.
The passage from isotropic to mixed smoothness is substantive rather
than notational. In fact, apart from the trivial \(L^2\) case, a
Sobolev space of dominating mixed smoothness does not coincide with
any isotropic Sobolev space:
\[
H_{\operatorname{mix}}^{(s_1,s_2)}(\mathbb T^2)
=
H^s(\mathbb T^2)
\quad\Longleftrightarrow\quad
s_1=s_2=s=0.
\]

Thus the mixed scale retains directional information that is lost
when the two Fourier variables are combined into a single isotropic
weight.

We first consider the constant-coefficient vector field
\[
L=\partial_t-c\partial_x,
\qquad
c=a+ib,\quad a,b\in\mathbb R.
\]

If \(b\neq0\) and
\(f\in
H_{\operatorname{mix}}^{(k_1,k_2)}(\mathbb T^2)\), we show in
Theorem~\ref{theo_regularity_cte} that every distributional solution
of \(Lu=f\) satisfies
\[
u\in
H_{\operatorname{mix}}^{
	(k_1+\theta,k_2+1-\theta)}
(\mathbb T^2),
\qquad
0\leq\theta\leq1.
\]

Thus a gain of one derivative can be distributed between the two
variables.

The situation changes when \(c=a\in\mathbb R\setminus\mathbb Q\).
If \(a\) has finite irrationality measure, then, for every
\(r>\mu(a)\),
\[
u\in
H_{\operatorname{mix}}^{
	(k_1+\theta,k_2+1-r-\theta)}
(\mathbb T^2),
\qquad
-r\leq\theta\leq1.
\]

The total shift is now \(1-r\), reflecting the small-divisor loss,
while \(\theta\) describes how this loss is distributed between the
two variables. The range of \(\theta\) is optimal for the
corresponding uniform multiplier estimates. We also prove sharpness
of the arithmetic threshold below the irrationality measure: if
\(r<\mu(a)\), Proposition~
\ref{prop:sharpness-real-constant-mixed} constructs data and a single
distributional solution for which the preceding mixed smoothness
conclusion fails simultaneously for every \(\theta\in\mathbb R\).
The critical value \(r=\mu(a)\) depends on finer arithmetic
properties.

The rational and Liouville cases are also considered, and in both
cases we obtain obstructions to universal mixed smoothness estimates
with any prescribed finite loss; see
Propositions~\ref{prop:rational-no-mixed-regularity} and
\ref{prop:liouville-no-mixed-regularity}.

We next turn to the real-valued variable-coefficient vector field
\[
L=\partial_t-a(t)\partial_x,
\qquad
a\in C^\infty(\mathbb T;\mathbb R),
\]
and denote its averaged coefficient by
\[
a_0
\coloneq
\frac{1}{2\pi}\int_0^{2\pi}a(t)\,dt.
\]

After taking the Fourier transform in \(x\), the equation reduces to a
family of periodic first-order ordinary differential equations
indexed by the \(x\)-frequency. Their small divisors are determined
by \(a_0\), whereas differentiation in \(t\) introduces additional
powers of the \(x\)-frequency. Estimating these periodic mode
equations directly, and combining integer-order estimates with
interpolation and duality, gives directional regularity estimates for
arbitrary real Sobolev orders.

More precisely, suppose that
\(a_0\in\mathbb R\setminus\mathbb Q\) has finite irrationality
measure, and let \(u\in\mathcal D'(\mathbb T^2)\) be a solution of
\(Lu=f\), with
\[
f\in
H_{\operatorname{mix}}^{(k_1,k_2)}(\mathbb T^2).
\]

Then Theorem~\ref{theo_regularity_real} gives, for every
\(r>\mu(a_0)\), every \(q\leq k_1\), and every
\(0\leq\theta\leq1\),
\[
u\in
H_{\operatorname{mix}}^{
	(q+\theta,\,
	k_2+1-r-|q|-\theta)}
(\mathbb T^2).
\]

The term \(|q|\) records the additional cost of estimating the
periodic mode equations at negative Sobolev orders. In particular,
when \(k_1\geq0\) and \(\theta=0\), this yields the simpler family
\[
u\in
H_{\operatorname{mix}}^{
	(\sigma,k_2+1-r-\sigma)}
(\mathbb T^2),
\qquad
0\leq\sigma\leq k_1+1.
\]

A second approach to the variable-coefficient problem is provided by
the periodic normal-form conjugation
\[
(\Psi_{\pm A}u)(t,x)
=
u(t,x\mp A(t)),
\qquad
A(t)
=
\int_0^t\bigl(a(\sigma)-a_0\bigr)\,d\sigma.
\]
Since \(a-a_0\) has zero mean, \(A\) is \(2\pi\)-periodic, and
\[
\Psi_A
(\partial_t-a(t)\partial_x)
=
(\partial_t-a_0\partial_x)\Psi_A.
\]

This conjugation provides a useful comparison between the mixed and
isotropic Sobolev scales. In the mixed smoothness scale,
Proposition~\ref{prop:conjugation-mixed-smoothness} gives
\[
\Psi_{\pm A}\colon
H_{\operatorname{mix}}^{(s_1,s_2)}(\mathbb T^2)
\longrightarrow
H_{\operatorname{mix}}^{
	(s_1,s_2-|s_1|)}
(\mathbb T^2).
\]

The loss in the second index comes from differentiating the
oscillatory factor \(e^{\pm i\xi A(t)}\) with respect to \(t\), which
produces powers of the \(x\)-frequency. Consequently, the mixed
smoothness estimates obtained through conjugation are dominated by
those given by the direct modal method. In this sense, the direct
analysis preserves more directional information.

By contrast, in the isotropic Sobolev scale, the operators
\(\Psi_{\pm A}\) are bounded automorphisms of
\(H^s(\mathbb T^2)\) for every \(s\in\mathbb R\). Hence the
constant-coefficient isotropic theory transfers to the
variable-coefficient vector field without an additional Sobolev loss.
In particular,
\[
f\in H^s(\mathbb T^2)
\quad\Longrightarrow\quad
u\in H^{s+1-r}(\mathbb T^2),
\qquad
r>\mu(a_0),
\]
whenever \(a_0\in\mathbb R\setminus\mathbb Q\) has finite
irrationality measure; see
Corollary~\ref{cor:isotropic-regularity-variable}.

This also yields a direct classical regularity consequence. If
\(f\in H^k(\mathbb T^2)\), with \(k>\mu(a_0)\), then
\[
u\in C^m(\mathbb T^2)
\qquad
\text{for every }m<k-\mu(a_0).
\]

By comparison, the classical regularity obtained from the direct
mixed smoothness estimates is guaranteed whenever
\[
3m-\frac12<k-\mu(a_0).
\]

Thus the normal-form argument gives the stronger classical threshold,
while the direct mixed smoothness estimates retain separate
information on the regularity in the two variables; see
Corollaries~\ref{cor:variable-real-classical-regularity} and
\ref{cor:isotropic-variable-classical-regularity}.

In the nonresonant regimes considered above, the regularity estimates
are accompanied by solvability and uniqueness. The natural
compatibility condition is \(\widehat f(0,0)=0\), and, under this
condition, there exists a unique distributional solution satisfying
\(\widehat u(0,0)=0\).

The paper is organized as follows.
Section~\ref{sec:preliminaries} recalls the basic properties of
Sobolev spaces of dominating mixed smoothness and their relation to
isotropic Sobolev spaces, while
Section~\ref{sec:fourier-reductions} develops the Fourier reductions
and small-divisor estimates used throughout the paper.
Section~\ref{sec:constant-coefficients} treats the constant-coefficient
problem, including mixed regularity, sharpness, and the rational and
Liouville obstructions.
Sections~\ref{sec:variable-real} and~\ref{sec:normal-form} then study
real-valued variable coefficients, first by the direct modal method
and subsequently through the normal-form conjugation, with isotropic,
classical, and arithmetic consequences.

%======================================
%======================================
\section{Sobolev Spaces of Dominating Mixed Smoothness and Auxiliary Results}
\label{sec:preliminaries}
%======================================
%======================================

In this section, we fix the notation and recall the basic properties
of periodic Sobolev spaces of dominating mixed smoothness, together
with the auxiliary estimates used throughout the paper.

We set \(\mathbb T\coloneq \mathbb R/2\pi\mathbb Z\), use
\((t,x)\in\mathbb T^2\) for the variables, and denote the Fourier
frequencies associated with \(t\) and \(x\) by \(\tau\) and \(\xi\),
respectively. For \(w\in\mathbb \R^n\), $n=1,2$, we write
\[
\langle w\rangle\coloneq (1+|w|^2)^{1/2}.
\]
In particular, 
\(\langle(\tau,\xi)\rangle = (1+\tau^2+\xi^2)^{1/2}\).

We denote by \(\mathbb N_0\) the set of nonnegative integers.
For nonnegative quantities \(A\) and \(B\), we write
\[
A\lesssim_{\alpha_1,\ldots,\alpha_N}B
\]
if there exists a constant \(C>0\), depending only on the displayed
parameters, such that \(A\leq CB\). The notation
\(A\gtrsim_{\alpha_1,\ldots,\alpha_N}B\) is defined analogously, and
we write
\[
A\asymp_{\alpha_1,\ldots,\alpha_N}B
\]
when both inequalities hold.

%=================================================
\subsection{Fourier analysis and Sobolev spaces of dominating mixed smoothness} \
%=================================================

For \(f\in L^1(\mathbb T^2)\), its Fourier coefficients are defined by
\[
\widehat f(\tau,\xi)
\coloneq 
\frac{1}{(2\pi)^2}
\int_{\mathbb T^2}
f(t,x)e^{-i(t\tau+x\xi)}\,dt\,dx,
\qquad
(\tau,\xi)\in\mathbb Z^2.
\]

The same notation is used for \(f\in\mathcal D'(\mathbb T^2)\), with
the Fourier coefficients defined by duality. Every distribution
\(f\in\mathcal D'(\mathbb T^2)\) admits the Fourier expansion
\[
f(t,x)
=
\sum_{(\tau,\xi)\in\mathbb Z^2}
\widehat f(\tau,\xi)e^{i(t\tau+x\xi)},
\]
with convergence in \(\mathcal D'(\mathbb T^2)\).

We shall use the standard characterizations of smooth functions and
distributions in terms of their Fourier coefficients. A sequence
\(\{a_{\tau,\xi}\}_{(\tau,\xi)\in\mathbb Z^2}\) is the sequence of
Fourier coefficients of a unique distribution on \(\mathbb T^2\) if
and only if it has at most polynomial growth, that is, if there exists
\(N\in\mathbb N_0\) such that
\[
|a_{\tau,\xi}|
\lesssim
\langle(\tau,\xi)\rangle^N,
\qquad
(\tau,\xi)\in\mathbb Z^2.
\]

Moreover, a distribution \(f\in\mathcal D'(\mathbb T^2)\) belongs to
\(C^\infty(\mathbb T^2)\) if and only if its Fourier coefficients are
rapidly decreasing, namely, if for every \(N\in\mathbb N_0\),
\[
|\widehat f(\tau,\xi)|
\lesssim_N
\langle(\tau,\xi)\rangle^{-N},
\qquad
(\tau,\xi)\in\mathbb Z^2.
\]

In the latter case, the Fourier series converges in
\(C^\infty(\mathbb T^2)\).

If \(f\in L^2(\mathbb T^2)\), then Plancherel's identity reads
\[
\|f\|_{L^2(\mathbb T^2)}^2
=
(2\pi)^2
\sum_{(\tau,\xi)\in\mathbb Z^2}
|\widehat f(\tau,\xi)|^2.
\]

For \(f\in L^1(\mathbb T^2)\), its partial Fourier coefficient with
respect to the \(x\)-variable is defined by
\[
\widehat f(t,\xi)
\coloneq 
\frac{1}{2\pi}
\int_{\mathbb T}
f(t,x)e^{-ix\xi}\,dx,
\qquad
\xi\in\mathbb Z.
\]

For a general distribution \(f\in\mathcal D'(\mathbb T^2)\), the
partial Fourier coefficient \(\widehat f(\cdot,\xi)\) is understood as
an element of \(\mathcal D'(\mathbb T)\) and is defined by duality.
Then
\[
f(t,x)
=
\sum_{\xi\in\mathbb Z}
\widehat f(t,\xi)e^{ix\xi},
\]
with convergence in \(\mathcal D'(\mathbb T^2)\). If
\(f\in C^\infty(\mathbb T^2)\), the convergence also holds in
\(C^\infty(\mathbb T^2)\). We use the same notation for full and
partial Fourier coefficients; the meaning will always be clear from
the displayed variables.

For \(s\in\mathbb R\), the isotropic Sobolev space
\(H^s(\mathbb T^2)\) consists of all
\(f\in\mathcal D'(\mathbb T^2)\) such that
\[
\|f\|_{H^s(\mathbb T^2)}^2
\coloneq 
\sum_{(\tau,\xi)\in\mathbb Z^2}
\langle(\tau,\xi)\rangle^{2s}
|\widehat f(\tau,\xi)|^2
<\infty.
\]

The one-dimensional Sobolev spaces are defined analogously, with norm
\[
\|v\|_{H^s(\mathbb T)}^2
\coloneq 
\sum_{\tau\in\mathbb Z}
\langle\tau\rangle^{2s}
|\widehat v(\tau)|^2.
\]

We next introduce the Sobolev scale of dominating mixed smoothness
used throughout the paper. In contrast with the isotropic Sobolev
scale, whose Fourier weight depends on the joint frequency
\((\tau,\xi)\), the spaces considered below are defined by product
weights, allowing the regularity in the two variables to be measured
independently.

For the classical theory of Sobolev spaces of dominating mixed
smoothness, including their tensor-product structure and equivalent
Fourier descriptions, see, for instance,
\cite{KuhnSickelUllrich2021,CobosKuhnSickel2016,
	Tem1993_book,Tem2018_book}.
Here we work with arbitrary real orders, defined directly through the
same product Fourier weights.

\begin{definition}
	Let \(s_1,s_2\in\mathbb R\). The anisotropic Sobolev space of dominating mixed smoothness
	\(H_{\operatorname{mix}}^{(s_1,s_2)}(\mathbb T^2)\)
	consists of all \(f\in\mathcal D'(\mathbb T^2)\) such that
	\begin{equation}\label{eq:def-mixed-Sobolev}
		\|f\|_{
			H_{\operatorname{mix}}^{(s_1,s_2)}
			(\mathbb T^2)
		}^2
		\coloneq
		\sum_{(\tau,\xi)\in\mathbb Z^2}
		\langle\tau\rangle^{2s_1}
		\langle\xi\rangle^{2s_2}
		|\widehat f(\tau,\xi)|^2
		<\infty.
	\end{equation}
\end{definition}

Equipped with the inner product associated with
\eqref{eq:def-mixed-Sobolev},
\(H_{\operatorname{mix}}^{(s_1,s_2)}(\mathbb T^2)\)
is a Hilbert space. Moreover, trigonometric polynomials are dense in
\(H_{\operatorname{mix}}^{(s_1,s_2)}(\mathbb T^2)\) for
every \(s_1,s_2\in\mathbb R\).

The product structure of the Fourier weight is the characteristic
feature of dominating mixed smoothness and distinguishes this scale
from the isotropic Sobolev scale. With the Fourier norms fixed above,
there is a canonical isometric identification
\[
H_{\operatorname{mix}}^{(s_1,s_2)}(\mathbb T^2)
\cong
H^{s_1}(\mathbb T)
\widehat\otimes_2
H^{s_2}(\mathbb T),
\]
where \(\widehat\otimes_2\) denotes the completed Hilbert tensor product.

For every
\(f\in H_{\operatorname{mix}}^{(s_1,s_2)}(\mathbb T^2)\),
it follows directly from the definition that
\begin{equation}\label{eq:mixed-partial-characterization}
	\|f\|_{
		H_{\operatorname{mix}}^{(s_1,s_2)}
		(\mathbb T^2)
	}^2
	=
	\sum_{\xi\in\mathbb Z}
	\langle\xi\rangle^{2s_2}
	\|\widehat f(\cdot,\xi)\|_{H^{s_1}(\mathbb T)}^2.
\end{equation}

%In particular, for every \(\xi\in\mathbb Z\),
%\begin{equation}\label{eq:partial-coefficient-Hs-bound}
%	\|\widehat f(\cdot,\xi)\|_{H^{s_1}(\mathbb T)}
%	\leq
%	\langle\xi\rangle^{-s_2}
%	\|f\|_{
%		H_{\operatorname{mix}}^{(s_1,s_2)}
%		(\mathbb T^2)
%	}.
%\end{equation}

The scale
\(H_{\operatorname{mix}}^{(s_1,s_2)}(\mathbb T^2)\)
is monotone in each index. More precisely, if
\(r_j\geq s_j\) for \(j=1,2\), then
\begin{equation*}
	H_{\operatorname{mix}}^{(r_1,r_2)}(\mathbb T^2)
	\hookrightarrow
	H_{\operatorname{mix}}^{(s_1,s_2)}(\mathbb T^2).
\end{equation*}

The Fourier characterizations recalled above also yield
\begin{equation*}
	C^\infty(\mathbb T^2)
	=
	\bigcap_{N\in\mathbb N_0}
	H_{\operatorname{mix}}^{(N,N)}(\mathbb T^2),
	\qquad
	\mathcal D'(\mathbb T^2)
	=
	\bigcup_{N\in\mathbb N_0}
	H_{\operatorname{mix}}^{(-N,-N)}(\mathbb T^2).
\end{equation*}

The sesquilinear pairing
\[
\langle f,g\rangle_0
\coloneq 
\sum_{(\tau,\xi)\in\mathbb Z^2}
\widehat f(\tau,\xi)
\overline{\widehat g(\tau,\xi)},
\]
initially defined for trigonometric polynomials, extends continuously to
\[
H_{\operatorname{mix}}^{(s_1,s_2)}(\mathbb T^2)
\times
H_{\operatorname{mix}}^{(-s_1,-s_2)}(\mathbb T^2).
\]

For each \(g\in H_{\operatorname{mix}}^{(-s_1,-s_2)}(\mathbb T^2)\),
the map \(f\longmapsto\langle f,g\rangle_0\) 
is a continuous linear functional on
\(H_{\operatorname{mix}}^{(s_1,s_2)}(\mathbb T^2)\),
and every continuous linear functional arises uniquely in this way.
Consequently,
\begin{equation*}
	\left(
	H_{\operatorname{mix}}^{(s_1,s_2)}(\mathbb T^2)
	\right)'
	\cong
	H_{\operatorname{mix}}^{(-s_1,-s_2)}(\mathbb T^2),
\end{equation*}
where the identification is conjugate-linear and isometric.

For a compatible couple of Banach spaces \((X_0,X_1)\), denote by
\([X_0,X_1]_\theta\), \(0<\theta<1\), the corresponding complex
interpolation space. The mixed Sobolev scale is stable under complex
interpolation. More precisely,
\begin{equation*}\label{eq:mixed-interpolation}
	\left[
	H_{\operatorname{mix}}^{(s_1,s_2)}(\mathbb T^2),
	H_{\operatorname{mix}}^{(r_1,r_2)}(\mathbb T^2)
	\right]_\theta
	=
	H_{\operatorname{mix}}^{
		((1-\theta)s_1+\theta r_1,\,
		(1-\theta)s_2+\theta r_2)
	}(\mathbb T^2),
\end{equation*}
with equivalence of norms. This follows from the standard complex
interpolation result for weighted \(L^2\)-spaces; see, for instance,
\cite[Section~5.5]{BerghLofstrom1976}.

In particular, boundedness estimates at two pairs of mixed orders can
be interpolated to obtain estimates at every intermediate pair.

%=================================================
\subsection{Comparison with isotropic Sobolev spaces and classical regularity} \
%=================================================

The relation between isotropic Sobolev spaces and Sobolev spaces of
dominating mixed smoothness is classical; see, for instance,
\cite{KuhnSickelUllrich2015}. 
For later use, we record the precise two-dimensional formulation for
arbitrary real indices and include, for completeness, the short proof,
which reduces to a direct comparison of the corresponding Fourier
weights.

\begin{proposition}\label{prop:isotropic-mixed-embeddings}
	Let \(s_1,s_2,\sigma,\rho\in\mathbb R\). Then:
	\begin{enumerate}
		\item
		\(
		H^\sigma(\mathbb T^2) \hookrightarrow
		H_{\operatorname{mix}}^{(s_1,s_2)}(\mathbb T^2)
		\quad\Longleftrightarrow\quad
		\sigma\geq\max\{s_1,s_2,s_1+s_2\};
		\)
		
		\item
		\(
		H_{\operatorname{mix}}^{(s_1,s_2)}(\mathbb T^2) \hookrightarrow
		H^\rho(\mathbb T^2) \quad\Longleftrightarrow\quad
		\rho\leq\min\{s_1,s_2,s_1+s_2\}.
		\)
	\end{enumerate}
\end{proposition}

\begin{proof}
	Set \(m\coloneq\min\{s_1,s_2,s_1+s_2\}\) and
	\(M\coloneq\max\{s_1,s_2,s_1+s_2\}\).
	We first observe that
	\begin{equation}\label{eq:mixed-isotropic-weight-comparison}
		\langle(\tau,\xi)\rangle^m
		\lesssim
		\langle\tau\rangle^{s_1}
		\langle\xi\rangle^{s_2}
		\lesssim
		\langle(\tau,\xi)\rangle^M.
	\end{equation}
	
	Indeed, if
	\(\langle\tau\rangle\geq\langle\xi\rangle\), then
	\(\langle(\tau,\xi)\rangle\asymp\langle\tau\rangle\) and
	\[
	\langle\tau\rangle^{\min\{s_1,s_1+s_2\}}
	\leq
	\langle\tau\rangle^{s_1}\langle\xi\rangle^{s_2}
	\leq
	\langle\tau\rangle^{\max\{s_1,s_1+s_2\}}.
	\]
	
	The complementary region
	\(\langle\xi\rangle\geq\langle\tau\rangle\) is analogous, with
	\(s_1\) and \(s_2\) interchanged. This proves
	\eqref{eq:mixed-isotropic-weight-comparison}.
	
	If \(\sigma\geq M\), the upper bound in
	\eqref{eq:mixed-isotropic-weight-comparison} immediately gives
	\(H^\sigma(\mathbb T^2)\hookrightarrow
	H_{\operatorname{mix}}^{(s_1,s_2)}(\mathbb T^2)\).
	Similarly, if \(\rho\leq m\), the lower bound gives
	\(H_{\operatorname{mix}}^{(s_1,s_2)}(\mathbb T^2)
	\hookrightarrow H^\rho(\mathbb T^2)\).
	
	For necessity, consider
	\[
	u_N^{(1)}(t,x)=e^{iNt},
	\qquad
	u_N^{(2)}(t,x)=e^{iNx},
	\qquad
	u_N^{(3)}(t,x)=e^{iN(t+x)}.
	\]
	
	If \(H^\sigma(\mathbb T^2)\hookrightarrow
	H_{\operatorname{mix}}^{(s_1,s_2)}(\mathbb T^2)\),
	applying the embedding to these three families and letting
	\(N\to\infty\) yields, respectively,
	\(s_1\leq\sigma\), \(s_2\leq\sigma\), and
	\(s_1+s_2\leq\sigma\). Hence \(\sigma\geq M\).
		
	Likewise, if
	\(H_{\operatorname{mix}}^{(s_1,s_2)}(\mathbb T^2)
	\hookrightarrow H^\rho(\mathbb T^2)\),
	the same three families give \(\rho\leq s_1\), \(\rho\leq s_2\),
	and \(\rho\leq s_1+s_2\). Hence \(\rho\leq m\).
\end{proof}

In the nonnegative range, Proposition~\ref{prop:isotropic-mixed-embeddings} yields
\[
H^{s_1+s_2}(\mathbb T^2)
\hookrightarrow
H_{\operatorname{mix}}^{(s_1,s_2)}(\mathbb T^2)
\hookrightarrow
H^{\min\{s_1,s_2\}}(\mathbb T^2),
\qquad
s_1,s_2\geq0.
\]

In particular,
\[
H^{2s}(\mathbb T^2)
\hookrightarrow
H_{\operatorname{mix}}^{(s,s)}(\mathbb T^2)
\hookrightarrow
H^s(\mathbb T^2),
\qquad s>0,
\]
whereas
\[
H^s(\mathbb T^2)
\hookrightarrow
H_{\operatorname{mix}}^{(s,s)}(\mathbb T^2)
\hookrightarrow
H^{2s}(\mathbb T^2),
\qquad s<0.
\]

For \(s\neq0\), both inclusions in the corresponding chains are
strict. Indeed, set
\[
a_N\coloneq
\langle N\rangle^{-\frac{3s}{2}-\frac12},
\]
and consider
\[
u_{\mathrm{ax}}(t,x)
=
\sum_{N\geq1}a_Ne^{iNt},
\qquad
u_{\mathrm{diag}}(t,x)
=
\sum_{N\geq1}a_Ne^{iN(t+x)}.
\]

A direct computation gives
\[
\begin{aligned}
	\|u_{\mathrm{ax}}\|_{H_{\operatorname{mix}}^{(s,s)}}^2
	\asymp
	\|u_{\mathrm{diag}}\|_{H^s}^2
	&\asymp
	\sum_{N\geq1}\langle N\rangle^{-s-1},
	\\
	\|u_{\mathrm{ax}}\|_{H^{2s}}^2
	\asymp
	\|u_{\mathrm{diag}}\|_{H_{\operatorname{mix}}^{(s,s)}}^2
	&\asymp
	\sum_{N\geq1}\langle N\rangle^{s-1}.
\end{aligned}
\]

The first series converges precisely when \(s>0\), and the second
precisely when \(s<0\). Thus both inclusions above are strict, and
\[
H_{\operatorname{mix}}^{(s,s)}(\mathbb T^2)
\neq
H^s(\mathbb T^2),
\qquad
s\neq0.
\]

For \(s=0\),
\[
H_{\operatorname{mix}}^{(0,0)}(\mathbb T^2)
=
H^0(\mathbb T^2)
=
L^2(\mathbb T^2)
\]
with equivalent norms.

We next record the corresponding classical regularity consequence.
Let
\[
f\in
H_{\operatorname{mix}}^{(s_1,s_2)}(\mathbb T^2),
\]
and let \(j_1,j_2\in\mathbb N_0\) satisfy
\[
s_1>j_1+\frac12,
\qquad
s_2>j_2+\frac12.
\]

For \(0\leq p\leq j_1\) and \(0\leq q\leq j_2\),
Cauchy--Schwarz inequality gives
\begin{align*}
	\sum_{(\tau,\xi)\in\mathbb Z^2}
	|\tau|^p|\xi|^q|\widehat f(\tau,\xi)|
	\leq
	\|f\|_{
		H_{\operatorname{mix}}^{(s_1,s_2)}
	}
	\times
	\left(
	\sum_{\tau\in\mathbb Z}
	\langle\tau\rangle^{-2(s_1-p)}
	\right)^{1/2}
	\left(
	\sum_{\xi\in\mathbb Z}
	\langle\xi\rangle^{-2(s_2-q)}
	\right)^{1/2}.
\end{align*}

Both series on the right-hand side converge. Hence, the Fourier series
with coefficients
\[
(i\tau)^p(i\xi)^q\widehat f(\tau,\xi)
\]
converges absolutely and uniformly on \(\mathbb T^2\). Its sum is a
continuous function representing the distributional derivative
\(\partial_t^p\partial_x^qf\). Therefore,
\[
\partial_t^p\partial_x^qf\in C^0(\mathbb T^2),
\qquad
0\leq p\leq j_1,
\quad
0\leq q\leq j_2.
\]

The preceding argument yields the following embedding for the mixed
smoothness scale.

\begin{corollary}\label{cor:mixed-Cm-embedding}
	Let \(m\in\mathbb N_0\) and let \(s_1,s_2\in\mathbb R\) satisfy
	\[
	s_1>m+\frac12,
	\qquad
	s_2>m+\frac12.
	\]
	Then
	\begin{equation*}
		H_{\operatorname{mix}}^{(s_1,s_2)}
		(\mathbb T^2)
		\hookrightarrow
		C^m(\mathbb T^2)
	\end{equation*}
	continuously. In fact, every
	\(f\in
	H_{\operatorname{mix}}^{(s_1,s_2)}(\mathbb T^2)\)
	satisfies
	\[
	\partial_t^p\partial_x^q f\in C^0(\mathbb T^2),
	\qquad
	0\leq p,q\leq m,
	\]
	and
	\[
	\max_{0\leq p,q\leq m}
	\|\partial_t^p\partial_x^q f\|_{C^0(\mathbb T^2)}
	\lesssim_{s_1,s_2,m}
	\|f\|_{
		H_{\operatorname{mix}}^{(s_1,s_2)}
		(\mathbb T^2)
	}.
	\]
\end{corollary}

\begin{proof}
	Take \(j_1=j_2=m\) in the preceding argument.
\end{proof}

For comparison, the usual isotropic Sobolev embedding gives
\begin{equation*}
	H^\sigma(\mathbb T^2)
	\hookrightarrow
	C^m(\mathbb T^2)
	\quad\text{if}\quad
	\sigma>m+1.
\end{equation*}

Notice that the preceding conclusion is stronger than the mere embedding
into \(C^m(\mathbb T^2)\): it gives continuity of every mixed derivative
\(\partial_t^p\partial_x^q f\) with \(0\leq p,q\leq m\), even when
\(p+q>m\).

%==========================================
%==========================================
\section{Fourier Reductions and Small-Divisor Estimates}
\label{sec:fourier-reductions}
%==========================================
%==========================================

This section collects the Fourier reductions and small-divisor
estimates that underlie the analysis of both constant-coefficient and
real-valued variable-coefficient vector fields.

%=================================================
\subsection{Full Fourier reduction for constant coefficients} \
%=================================================

Let
\[
L_c\coloneq \partial_t-c\partial_x,
\quad c=a+ib, \quad a,b\in\mathbb R.
\]

For \(u,f\in\mathcal D'(\mathbb T^2)\), taking Fourier coefficients
in both variables, the equation \(L_cu=f\) becomes
\begin{equation}\label{eq:Fourier-equation-constant}
	i(\tau-c\xi)\widehat u(\tau,\xi)
	=
	\widehat f(\tau,\xi),
	\qquad
	(\tau,\xi)\in\mathbb Z^2.
\end{equation}

The resonance set of \(L_c\) is
\[
\mathcal R_c
\coloneq 
\left\{
(\tau,\xi)\in\mathbb Z^2:
\tau-c\xi=0
\right\}.
\]

If \(c\notin\mathbb R\) or
\(c\in\mathbb R\setminus\mathbb Q\), then
\(\mathcal R_c=\{(0,0)\}\). If
\(c=p_0/q_0\in\mathbb Q\), where
\(p_0\in\mathbb Z\), \(q_0\in\mathbb N\), and
\(\gcd(p_0,q_0)=1\), then
\(\mathcal R_c=\{(mp_0,mq_0):m\in\mathbb Z\}\).

Equation \eqref{eq:Fourier-equation-constant} imposes the
compatibility conditions
\[
\widehat f(\tau,\xi)=0,
\qquad
(\tau,\xi)\in\mathcal R_c.
\]

In particular, \(\widehat f(0,0)=0\), which is equivalent to the
zero-mean condition for \(f\).

Away from the resonance set, the Fourier coefficients of every
solution are uniquely determined by
\begin{equation*}
	\widehat u(\tau,\xi)
	=
	\frac{\widehat f(\tau,\xi)}
	{i(\tau-c\xi)},
	\qquad
	(\tau,\xi)\notin\mathcal R_c.
\end{equation*}

%=================================================
\subsection{Lower bounds for constant-coefficient symbols} \
%=================================================

When \(b\neq0\), the symbol \(\tau-c\xi\) is elliptic as a real-linear
function of \((\tau,\xi)\), and no small-divisor phenomenon occurs.
For real irrational coefficients, on the other hand, the small-divisor
behavior of the symbol is governed by Diophantine approximation.
Following \cite{KowKir2026-jfa}, we recall the relevant arithmetic
quantity.

\begin{definition} \label{def_complete_irrationality}
	Let \(a\in\mathbb R\setminus\mathbb Q\). Its irrationality measure
	is defined by
	\begin{equation} \label{def_irrationality}
		\mu(a)
		\coloneq 
		\sup
		\left\{
		\nu>0:
		\left|a-\frac{p}{q}\right|
		<
		q^{-\nu}
		\text{ for infinitely many }
		(p,q)\in\mathbb Z\times\mathbb N
		\right\}.
	\end{equation}
\end{definition}

Every irrational number satisfies \(\mu(a)\geq2\), and \(a\) is a
Liouville number if and only if \(\mu(a)=\infty\). If
\(\mu(a)<\infty\), then, for every \(r>\mu(a)\), there exists a
constant \(C_{a,r}>0\) such that
\begin{equation}\label{eq:diophantine-lower-bound}
	|\tau-a\xi|
	\geq
	C_{a,r}\langle\xi\rangle^{1-r},
	\qquad
	(\tau,\xi)\in\mathbb Z^2,
	\quad
	\xi\neq0.
\end{equation}

Indeed, by the definition of \(\mu(a)\), for every
\(r>\mu(a)\) the inequality
\[
\left|a-\frac{p}{q}\right|<q^{-r}
\]
has only finitely many solutions
\((p,q)\in\mathbb Z\times\mathbb N\). For \(\xi\neq0\), taking
\(q=|\xi|\) and \(p=\operatorname{sgn}(\xi)\tau\), we have
\[
|\tau-a\xi|
=
|\xi|
\left|
a-\frac{p}{q}
\right|.
\]

Hence, for all but finitely many \(\xi\neq0\) and every
\(\tau\in\mathbb Z\),
\[
|\tau-a\xi|
\geq
|\xi|^{1-r}.
\]

For each of the finitely many remaining values of \(\xi\),
irrationality of \(a\) gives
\[
\inf_{\tau\in\mathbb Z}|\tau-a\xi|>0.
\]

After adjusting the constant and replacing \(|\xi|\) by
\(\langle\xi\rangle\), we obtain
\eqref{eq:diophantine-lower-bound}.

\begin{remark}\label{rem:endpoint-irrationality-measure}
	The definition of \(\mu(a)\) guarantees
	\eqref{eq:diophantine-lower-bound} for every \(r>\mu(a)\), but
	does not by itself determine whether the same estimate holds at
	the endpoint \(r=\mu(a)\). For a general irrational number, the
	endpoint estimate need not follow from the value of \(\mu(a)\)
	alone.
	
	For quadratic irrational numbers, however, Liouville's approximation
	theorem for algebraic numbers yields 
	\eqref{eq:diophantine-lower-bound} with \(r=\mu(a)=2\).
	Consequently, all the results proved below for
	\(r>\mu(a)\) remain valid with \(r=2\) when \(a\) is a quadratic
	irrational. More generally, the same conclusion holds with
	\(r=\mu(a)\) whenever
	\eqref{eq:diophantine-lower-bound} is valid when the supremum in 
	\eqref{def_irrationality} is attained.
\end{remark}

The following proposition establishes the symbol estimates underlying
the constant-coefficient regularity results in the mixed smoothness
scale. The parameter \(\theta\) describes how the corresponding gain
or loss of regularity is distributed between the \(t\)- and
\(x\)-directions.

\begin{proposition}\label{prop:mixed-symbol-lower-bounds}
	Let \(c=a+ib\in\mathbb C\).
	
	\begin{enumerate}
		\item[(i)]
		If \(b\neq0\), then
		\begin{equation}\label{diof_ineq_imaginary}
			|\tau-c\xi|
			\gtrsim_{a,b}
			\langle\tau\rangle^\theta
			\langle\xi\rangle^{1-\theta},
			\qquad
			0\leq\theta\leq1,
		\end{equation}
		for every
		\((\tau,\xi)\in\mathbb Z^2\setminus\{(0,0)\}\).
		
		\item[(ii)]
		If \(b=0\) and \(a\in\mathbb R\setminus\mathbb Q\)
		has finite irrationality measure, then, for every
		\(r>\mu(a)\),
		\begin{equation}\label{diof_ineq_real}
			|\tau-a\xi|
			\gtrsim_{a,r}
			\langle\tau\rangle^\theta
			\langle\xi\rangle^{1-r-\theta},
			\qquad
			-r\leq\theta\leq1,
		\end{equation}
		for every
		\((\tau,\xi)\in\mathbb Z^2\setminus\{(0,0)\}\).
	\end{enumerate}
	
	In each case, the implicit constant can be chosen independently
	of \(\theta\) in the indicated interval.
\end{proposition}

\begin{proof}
	Assume first that \(b\neq0\). The real-linear map
	\(T(\tau,\xi)\coloneq(\tau-a\xi,-b\xi)\) is invertible. Hence
	\[
	|\tau-c\xi|
	=
	|T(\tau,\xi)|
	\gtrsim_{a,b}
	(\tau^2+\xi^2)^{1/2}.
	\]
	
	For
	\((\tau,\xi)\in\mathbb Z^2\setminus\{(0,0)\}\), it follows that
	\[
	|\tau-c\xi|
	\gtrsim_{a,b}\langle\tau\rangle,
	\qquad
	|\tau-c\xi|
	\gtrsim_{a,b}\langle\xi\rangle.
	\]
	
	Thus, for \(0\leq\theta\leq1\),
	\[
	|\tau-c\xi|
	\geq
	C_1^\theta C_2^{1-\theta}
	\langle\tau\rangle^\theta
	\langle\xi\rangle^{1-\theta}
	\geq
	\min\{C_1,C_2\}
	\langle\tau\rangle^\theta
	\langle\xi\rangle^{1-\theta}.
	\]
	This proves \emph{(i)}, with an implicit constant independent of
	\(\theta\).
	
	We now prove \emph{(ii)}. Fix \(r>\mu(a)\geq2\). We first consider
	the endpoint \(\theta=1\) in \eqref{diof_ineq_real}. Let
	\((\tau,\xi)\in\mathbb Z^2\setminus\{(0,0)\}\).
	If 	\(|\tau|\geq2|a|\,|\xi|\),
	then \(\tau\neq0\) and
	\[
	|\tau-a\xi|
	\geq
	\frac{|\tau|}{2}
	\gtrsim
	\langle\tau\rangle
	\geq
	\langle\tau\rangle\langle\xi\rangle^{-r}.
	\]
	
	If \(|\tau|<2|a|\,|\xi|\), then \(\xi\neq0\) and
	\(\langle\tau\rangle\lesssim_a\langle\xi\rangle\). Hence, by
	\eqref{eq:diophantine-lower-bound},
	\[
	\langle\tau\rangle\langle\xi\rangle^{-r}
	\lesssim_a
	\langle\xi\rangle^{1-r}
	\lesssim_{a,r}
	|\tau-a\xi|.
	\]
	
	Combining the two cases, we obtain
	\begin{equation}\label{eq:real-endpoint-theta-one}
		|\tau-a\xi|
		\gtrsim_{a,r}
		\langle\tau\rangle\langle\xi\rangle^{-r}.
	\end{equation}
	
	We next consider the other endpoint, \(\theta=-r\). If \(|a|\,|\xi|\geq2|\tau|\),
	then \(\xi\neq0\) and
	\[
	|\tau-a\xi|
	\geq
	\frac{|a||\xi|}{2}
	\gtrsim_a
	\langle\tau\rangle^{-r}\langle\xi\rangle.
	\]
	
	If \(|a|\,|\xi|<2|\tau|\), then either \(\xi=0\), in which case
	\[
	\langle\tau\rangle^{-r}\langle\xi\rangle
	=
	\langle\tau\rangle^{-r}
	\leq
	|\tau|
	=
	|\tau-a\xi|,
	\]
	or \(\xi\neq0\), in which case
	\(\langle\xi\rangle\lesssim_a\langle\tau\rangle\). Since
	\(r>1\), \eqref{eq:diophantine-lower-bound} gives
	\[
	|\tau-a\xi|
	\gtrsim_{a,r}
	\langle\xi\rangle^{1-r}
	\gtrsim_{a,r}
	\langle\tau\rangle^{1-r}
	\gtrsim_a
	\langle\tau\rangle^{-r}\langle\xi\rangle.
	\]
	
	Thus
	\begin{equation}\label{eq:real-endpoint-theta-minus-r}
		|\tau-a\xi|
		\gtrsim_{a,r}
		\langle\tau\rangle^{-r}\langle\xi\rangle.
	\end{equation}
	
	Finally, let \(\theta\in[-r,1]\) and set
	\[
	\lambda
	\coloneq
	\frac{\theta+r}{r+1}.
	\]
	
	Then \(\lambda\in[0,1], \theta=\lambda-r(1-\lambda)\), and
	\(1-r-\theta=-r\lambda+(1-\lambda)\). 		
	Taking the weighted geometric mean of
	\eqref{eq:real-endpoint-theta-one} and
	\eqref{eq:real-endpoint-theta-minus-r}, we obtain
	\begin{equation*}
		|\tau-a\xi|
		\gtrsim_{a,r}
		\left(
		\langle\tau\rangle\langle\xi\rangle^{-r}
		\right)^\lambda
		\left(
		\langle\tau\rangle^{-r}\langle\xi\rangle
		\right)^{1-\lambda}
		=
		\langle\tau\rangle^\theta
		\langle\xi\rangle^{1-r-\theta}.
	\end{equation*}
	
	As in part~\emph{(i)}, the implicit constant can be chosen
	uniformly for \(\theta\in[-r,1]\).
\end{proof}

%=================================================
\subsection{Partial Fourier reduction for variable coefficients}  \
%=================================================

Let \(a\in C^\infty(\mathbb T;\mathbb R)\), and denote its average by
\[
a_0
\coloneq
\frac{1}{2\pi}\int_0^{2\pi}a(t)\,dt.
\]

Taking the Fourier transform with respect to the \(x\)-variable, the
equation
\[
Lu=f,
\qquad
L=\partial_t-a(t)\partial_x,
\]
becomes
\begin{equation}\label{eq:variable-partial-Fourier-equation}
	\partial_t\widehat u(t,\xi)
	-i\xi a(t)\widehat u(t,\xi)
	=
	\widehat f(t,\xi),
	\qquad
	\xi\in\mathbb Z.
\end{equation}

For \(\xi=0\), equation
\eqref{eq:variable-partial-Fourier-equation} reduces to
\(\partial_t\widehat u(t,0)=\widehat f(t,0)\). 
Since \(\widehat u(\cdot,0)\) is periodic, necessarily
\[
\int_0^{2\pi}\widehat f(t,0)\,dt=0,
\]
or, equivalently, \(\widehat f(0,0)=0\).

Under this compatibility condition, the zero-mode solutions are given by
\[
\widehat u(t,0)
=
\lambda+\int_0^t\widehat f(s,0)\,ds,
\qquad
\lambda\in\mathbb C.
\]

For the nonzero Fourier modes, we use the following classical
representation formulas.

\begin{lemma}\label{lem:periodic-mode-equation}
	Let \(\xi\in\mathbb Z\setminus\{0\}\) and
	\(g\in L^1(\mathbb T)\). If \(\xi a_0\notin\mathbb Z\), 
	then the equation
	\begin{equation}\label{eq:abstract-periodic-mode}
		v'(t)-i\xi a(t)v(t)=g(t),
		\qquad t\in\mathbb T,
	\end{equation}
	has a unique periodic distributional solution. This solution has an
	absolutely continuous periodic representative and is given by
	\begin{equation}\label{eq:general-variable-solution-backward}
		v(t)
		=
		\frac{1}{1-e^{2\pi i\xi a_0}}
		\int_0^{2\pi}
		\exp\left(
		i\xi\int_{t-s}^{t}a(r)\,dr
		\right)
		g(t-s)\,ds,
	\end{equation}
	or, equivalently, by
	\begin{equation}\label{eq:general-variable-solution-forward}
		v(t)
		=
		\frac{1}{e^{-2\pi i\xi a_0}-1}
		\int_0^{2\pi}
		\exp\left(
		-i\xi\int_t^{t+s}a(r)\,dr
		\right)
		g(t+s)\,ds.
	\end{equation}
\end{lemma}

The proof is immediate: multiplying
\eqref{eq:abstract-periodic-mode} by the integrating factor
and imposing periodicity yields
\eqref{eq:general-variable-solution-backward} and
\eqref{eq:general-variable-solution-forward}. Since
\(g\in L^1(\mathbb T)\), the resulting solution has an absolutely
continuous periodic representative, and the condition
\(\xi a_0\notin\mathbb Z\) guarantees uniqueness.

When \(f\in H_{\operatorname{mix}}^{(k_1,k_2)}(\mathbb T^2)\),
with \(k_1\geq0\), 	
\eqref{eq:mixed-partial-characterization} implies that
\[
\widehat f(\cdot,\xi)
\in H^{k_1}(\mathbb T)
\subset L^2(\mathbb T)
\subset L^1(\mathbb T),
\qquad
\xi\in\mathbb Z.
\]

Thus the preceding representation formulas apply directly in this
range. Estimates for arbitrary real Sobolev orders will later be
obtained by interpolation and duality.

For \(\xi\neq0\), the regularity problem is therefore reduced to
estimating the small-divisor factors
\[
\left|1-e^{\pm2\pi i\xi a_0}\right|^{-1}.
\]
The required estimates are established in the next subsection.

%=================================================
\subsection{Small divisors for real-valued coefficients} \
%=================================================

Since \(a\) is real-valued, the exponential factors in
\eqref{eq:general-variable-solution-backward} and
\eqref{eq:general-variable-solution-forward} have absolute value one. Hence
the relevant small divisors are
\(1-e^{\pm2\pi i\xi a_0}\).

If \(a_0=p_0/q_0\in\mathbb Q\) is written in lowest terms, the
nonzero resonant modes are precisely
\[
\xi=mq_0,
\qquad
m\in\mathbb Z\setminus\{0\}.
\]

If \(a_0\notin\mathbb Q\), all nonzero modes are nonresonant, but the
corresponding denominators may become arbitrarily small as
\(|\xi|\to\infty\). Their decay is controlled by the Diophantine
properties of \(a_0\).

\begin{lemma}\label{lemma_exp_diof}
	For every \(y\in\mathbb R\),
	\begin{equation}\label{eq:exponential-distance}
		4\operatorname{dist}(y,\mathbb Z)
		\leq
		|1-e^{\pm2\pi iy}|
		\leq
		2\pi\operatorname{dist}(y,\mathbb Z),
	\end{equation}
	where 
	\(\operatorname{dist}(y,\mathbb Z)
	\coloneq  \inf_{\ell\in\mathbb Z}|y-\ell|\).
	
	Consequently, if
	\(a_0\in\mathbb R\setminus\mathbb Q\) has finite irrationality
	measure, then, for every \(r>\mu(a_0)\),
	\begin{equation}\label{eq:exponential-diophantine-bound}
		|1-e^{\pm2\pi i\xi a_0}|
		\gtrsim_{a_0,r}
		\langle\xi\rangle^{1-r},
		\qquad
		\xi\in\mathbb Z\setminus\{0\}.
	\end{equation}
\end{lemma}

\begin{proof}
	Let \(d=\operatorname{dist}(y,\mathbb Z)\). Since
	\(0\leq d\leq1/2\), then 
	\(|1-e^{\pm2\pi iy}| = 2\sin(\pi d)\). Moreover
	\(2d\leq\sin(\pi d)\leq\pi d\), which proves
	\eqref{eq:exponential-distance}.
	
	For \(\xi\neq0\), choose \(\tau\in\mathbb Z\) such that
	\(|\tau-a_0\xi| = \operatorname{dist}(a_0\xi,\mathbb Z)\).
	Then \eqref{eq:diophantine-lower-bound} and
	\eqref{eq:exponential-distance} yield
	\eqref{eq:exponential-diophantine-bound}.
\end{proof}

%==========================================
%==========================================
\section{Constant-Coefficient Vector Fields}
\label{sec:constant-coefficients}
%==========================================
%==========================================

We now consider the constant-coefficient vector field
\[
L\coloneq \partial_t-c\,\partial_x,
\quad
c=a+ib,\quad a,b\in\mathbb R,
\]
and study the regularity and solvability of \(Lu=f\) in Sobolev
spaces of dominating mixed smoothness.

%=================================================
\subsection{Mixed regularity and solvability} \
%=================================================

The following theorem gives the basic mixed smoothness estimates. In
the nonreal case, a gain of one derivative can be distributed between
the two variables, whereas in the real irrational case the
corresponding net shift is \(1-r\), for every \(r>\mu(a)\).

\begin{theorem}\label{theo_regularity_cte}
	Let \(k_1,k_2\in\mathbb R\), and suppose that
	\(u\in\mathcal D'(\mathbb T^2)\) is a distributional solution of
	\[
	Lu=f,
	\qquad
	f\in
	H_{\operatorname{mix}}^{(k_1,k_2)}(\mathbb T^2).
	\]
	Then the following statements hold.
	\begin{enumerate}
		\item[(i)]
		If \(b\neq0\), then, for every \(\theta\in[0,1]\),
		\begin{equation*}
			u\in
			H_{\operatorname{mix}}^{
				(k_1+\theta,k_2+1-\theta)
			}(\mathbb T^2).
		\end{equation*}
		Moreover,
		\begin{equation*}
			\left\|
			u-\widehat u(0,0)
			\right\|_{
				H_{\operatorname{mix}}^{
					(k_1+\theta,k_2+1-\theta)
				}(\mathbb T^2)
			}
			\lesssim_{a,b}
			\|f\|_{
				H_{\operatorname{mix}}^{(k_1,k_2)}
				(\mathbb T^2)
			},
		\end{equation*}
		with an implicit constant independent of \(\theta\).
		
		\item[(ii)]
		If \(b=0\) and
		\(a\in\mathbb R\setminus\mathbb Q\) has finite irrationality
		measure \(\mu(a)\), then, for every \(r>\mu(a)\) and every
		\(\theta\in[-r,1]\),
		\begin{equation*}
			u\in
			H_{\operatorname{mix}}^{
				(k_1+\theta,k_2+1-r-\theta)
			}(\mathbb T^2).
		\end{equation*}
		Moreover,
		\begin{equation*}
			\left\|
			u-\widehat u(0,0)
			\right\|_{
				H_{\operatorname{mix}}^{
					(k_1+\theta,k_2+1-r-\theta)
				}(\mathbb T^2)
			}
			\lesssim_{a,r}
			\|f\|_{
				H_{\operatorname{mix}}^{(k_1,k_2)}
				(\mathbb T^2)
			},
		\end{equation*}
		with an implicit constant independent of \(\theta\).
	\end{enumerate}
\end{theorem}

\begin{proof}
	Set \(v\coloneq u-\widehat u(0,0)\). Under either hypothesis, the
	resonance set is \(\{(0,0)\}\), so that
	\(\widehat f(0,0)=0\) and, for
	\((\tau,\xi)\neq(0,0)\),
	\[
	\widehat v(\tau,\xi)
	=
	\frac{\widehat f(\tau,\xi)}
	{i(\tau-c\xi)}.
	\]
	
	Set
	\[
	\beta\coloneq
	\begin{cases}
		1-\theta,
		& b\neq0,\quad 0\leq\theta\leq1,\\
		1-r-\theta,
		& b=0,\quad r>\mu(a),\quad -r\leq\theta\leq1.
	\end{cases}
	\]
	By Proposition~\ref{prop:mixed-symbol-lower-bounds},
	\[
	\langle\tau\rangle^{k_1+\theta}
	\langle\xi\rangle^{k_2+\beta}
	|\widehat v(\tau,\xi)|
	\lesssim
	\langle\tau\rangle^{k_1}
	\langle\xi\rangle^{k_2}
	|\widehat f(\tau,\xi)|,
	\]
	with an implicit constant uniform in \(\theta\) in the indicated
	intervals. Squaring and summing over \(\mathbb Z^2\) gives the
	asserted estimates. Since
	\(\beta=1-\theta\) in case~\emph{(i)} and
	\(\beta=1-r-\theta\) in case~\emph{(ii)}, the corresponding
	regularity statements follow.
\end{proof}

\begin{remark}
	If \(a\in\mathbb R\setminus\mathbb Q\) has finite irrationality
	measure, then the finite Sobolev loss in
	Theorem~\ref{theo_regularity_cte} is absorbed by smooth data. Hence
	\[
	f\in C^\infty(\mathbb T^2)
	\quad\Longrightarrow\quad
	u\in C^\infty(\mathbb T^2).
	\]
	
	At the qualitative \(C^\infty\)-level, this recovers the classical
	global hypoellipticity result for constant vector fields on the
	torus; see \cite{GW1972_pams,GW1973_top}. The quantitative Sobolev 
	estimates above refine this qualitative statement by identifying the 
	finite loss in terms of the irrationality measure \(\mu(a)\).
\end{remark}

\begin{corollary}\label{coro_existence_cte}
	Let \(k_1,k_2\in\mathbb R\), and let
	\(f\in H_{\operatorname{mix}}^{(k_1,k_2)}(\mathbb T^2)\) 
	satisfy \(\widehat f(0,0)=0\).
	Then the following statements hold.
	\begin{enumerate}
		\item[(i)]
		If \(b\neq0\), there exists a unique distributional solution
		\(u\) of \(Lu=f\) satisfying \(\widehat u(0,0)=0\).
		Moreover, for every \(\theta\in[0,1]\),
		\[
		u\in
		H_{\operatorname{mix}}^{
			(k_1+\theta,k_2+1-\theta)
		}(\mathbb T^2).
		\]
		
		\item[(ii)]
		If \(b=0\) and
		\(a\in\mathbb R\setminus\mathbb Q\) has finite irrationality
		measure \(\mu(a)\), there exists a unique distributional solution
		\(u\) of \(Lu=f\) satisfying \(\widehat u(0,0)=0\).
		Moreover, for every \(r>\mu(a)\) and every
		\(\theta\in[-r,1]\),
		\[
		u\in
		H_{\operatorname{mix}}^{
			(k_1+\theta,k_2+1-r-\theta)
		}(\mathbb T^2).
		\]
	\end{enumerate}
\end{corollary}

\begin{proof}
	Set \(\widehat u(0,0)=0\) and, for
	\((\tau,\xi)\neq(0,0)\), define
	\[
	\widehat u(\tau,\xi)
	=
	\frac{\widehat f(\tau,\xi)}
	{i(\tau-c\xi)}.
	\]
	Under either hypothesis, the denominator does not vanish away from
	the origin. The estimates used in the proof of
	Theorem~\ref{theo_regularity_cte} show that \(u\) belongs to all
	the spaces asserted above, and hence defines a distributional
	solution of \(Lu=f\).
	
	If \(v\) is another distributional solution satisfying
	\(\widehat v(0,0)=0\), then the Fourier equation implies
	\(\widehat u(\tau,\xi)=\widehat v(\tau,\xi)\) for every
	\((\tau,\xi)\neq(0,0)\). Since
	\(\widehat u(0,0)=\widehat v(0,0)=0\), we conclude that \(u=v\).
\end{proof}

\begin{remark}\label{rem:admissible-mixed-region-constant}
	In the real irrational case, the estimates of
	Theorem~\ref{theo_regularity_cte} admit a simple geometric
	interpretation. For each \(r>\mu(a)\), the regularity indices
	\((s_1,s_2)\) lie on the segment
	\[
	s_1+s_2=k_1+k_2+1-r,
	\qquad
	k_1-r\leq s_1\leq k_1+1.
	\]
	
	By monotonicity of the mixed smoothness scale and by taking the
	union over \(r>\mu(a)\), every pair satisfying
	\[
	s_1\leq k_1+1,\qquad
	s_2\leq k_2+1,\qquad
	s_1+s_2<k_1+k_2+1-\mu(a)
	\]
	is admissible.
	
	The resulting region is illustrated below.
	
	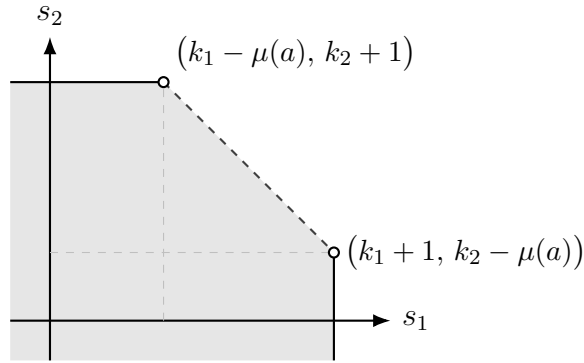
\begin{figure}[ht]
		\centering
		\begin{tikzpicture}[scale=0.75, >=Latex]
			
			%-------------------------------------------------
			% Parameters (edit these to adjust the picture)
			%-------------------------------------------------
			\pgfmathsetmacro{\kone}{4.0}
			\pgfmathsetmacro{\ktwo}{3.2}
			\pgfmathsetmacro{\mua}{2.0}
			
			\pgfmathsetmacro{\xmin}{-0.7}
			\pgfmathsetmacro{\xmax}{6.0}
			\pgfmathsetmacro{\ymin}{-0.7}
			\pgfmathsetmacro{\ymax}{5.0}
			
			% Key points
			\pgfmathsetmacro{\xA}{\kone-\mua}
			\pgfmathsetmacro{\yA}{\ktwo+1}
			\pgfmathsetmacro{\xB}{\kone+1}
			\pgfmathsetmacro{\yB}{\ktwo-\mua}
			
			%-------------------------------------------------
			% Shaded region:
			% s_1 <= k_1+1, s_2 <= k_2+1, s_1+s_2 < k_1+k_2+1-\mu(a)
			%-------------------------------------------------
			\fill[gray!20]
			(\xmin,\ymin) --
			(\xmin,\yA) --
			(\xA,\yA) --
			(\xB,\yB) --
			(\xB,\ymin) -- cycle;
			
			% Axes
			% \draw[thick] (\xmin,0) -- (\xmax,0);
			% \draw[thick] (0,\ymin) -- (0,\ymax);
			\draw[->, thick] (\xmin,0) -- (\xmax,0) node[right] {$s_1$};     \draw[->, thick] (0,\ymin) -- (0,\ymax) node[above] {$s_2$};

			% Solid boundary pieces (included parts)
			\draw[black, thick] (\xmin,\yA) -- (\xA,\yA);
			\draw[black, thick] (\xB,\ymin) -- (\xB,\yB);
			
			% Open boundary piece (not included)
			\draw[gray!50!black, dashed, thick] (\xA,\yA) -- (\xB,\yB);
			
			% Optional guide lines
			\draw[gray!50, dashed] (\xA,0) -- (\xA,\yA);
			\draw[gray!50, dashed] (0,\yB) -- (\xB,\yB);
			
			% Marked points
			\draw[black, thick, fill=white] (\xA,\yA) circle (2.5pt);
			\draw[black, thick, fill=white] (\xB,\yB) circle (2.5pt);
			
			\node[black, above right] at (\xA,\yA)
			{\small $\bigl(k_1-\mu(a),\,k_2+1\bigr)$};
			
			\node[black, right] at (\xB,\yB)
			{\small $\bigl(k_1+1,\,k_2-\mu(a)\bigr)$};
			
		\end{tikzpicture}
		\caption{\small
			Admissible mixed regularity region for the solution \(u\) of
			\(Lu=f\), with \(f\in\Hmix{(k_1,k_2)}\), in the real irrational
			constant-coefficient case. The dashed segment corresponds to the
			endpoint \(r=\mu(a)\), which is not included in general.}
	\end{figure}
	The slanted boundary
	\[
	s_1+s_2=k_1+k_2+1-\mu(a)
	\]
	is not included in general. If the Diophantine lower bound holds
	at the endpoint \(r=\mu(a)\), then the corresponding boundary
	segment is included as well.
\end{remark}

%=================================================
\subsection{Sharpness and isotropic comparison} \
%=================================================

The ranges of \(\theta\) in
Proposition~\ref{prop:mixed-symbol-lower-bounds} are optimal for the
corresponding uniform multiplier estimates, as can be seen by
restricting the frequencies to the coordinate axes.

If \(b\neq0\), an estimate of the form
\[
|\tau-c\xi|
\gtrsim_{a,b}
\langle\tau\rangle^\theta
\langle\xi\rangle^{1-\theta}
\]
requires \(\theta\leq1\) by taking \(\xi=0\), and
\(\theta\geq0\) by taking \(\tau=0\). Hence
\(0\leq\theta\leq1\) is optimal, with endpoint estimates
\[
u\in
H_{\operatorname{mix}}^{(k_1+1,k_2)}(\mathbb T^2)
\cap
H_{\operatorname{mix}}^{(k_1,k_2+1)}(\mathbb T^2).
\]

Similarly, in the real irrational case, an estimate of the form
\[
|\tau-a\xi|
\gtrsim_{a,r}
\langle\tau\rangle^\theta
\langle\xi\rangle^{1-r-\theta}
\]
requires \(\theta\leq1\) and \(\theta\geq-r\). Thus
\(-r\leq\theta\leq1\) is optimal for this family of uniform
multiplier estimates, and the endpoints give
\[
u\in
H_{\operatorname{mix}}^{(k_1+1,k_2-r)}(\mathbb T^2)
\cap
H_{\operatorname{mix}}^{(k_1-r,k_2+1)}(\mathbb T^2).
\]

The preceding discussion concerns the optimal range of \(\theta\) for
a fixed value of \(r\). We next show that, in the real irrational case,
the Diophantine threshold is also sharp below \(\mu(a)\). The critical
case \(r=\mu(a)\) depends on finer arithmetic properties, as explained
in Remark~\ref{rem:endpoint-irrationality-measure}.

\begin{proposition}\label{prop:sharpness-real-constant-mixed}
	Let \(a\in\mathbb R\setminus\mathbb Q\) have finite irrationality
	measure, and let \(L=\partial_t-a\partial_x\).
	
	If \(r<\mu(a)\), then, for every \(k_1,k_2\in\mathbb R\), there
	exist
	\[
	f\in
	H_{\operatorname{mix}}^{(k_1,k_2)}(\mathbb T^2),
	\qquad
	u\in\mathcal D'(\mathbb T^2),
	\]
	such that \(Lu=f\) and, for every \(\theta\in\mathbb R\),
	\[
	u\notin
	H_{\operatorname{mix}}^{
		(k_1+\theta,k_2+1-r-\theta)}
	(\mathbb T^2).
	\]
	
	In particular, the conclusion of
	Theorem~\ref{theo_regularity_cte} fails simultaneously for every
	\(\theta\in[-r,1]\).
\end{proposition}

\begin{proof}
	Choose \(r'\) such that \(\max\{r,1\}<r'<\mu(a)\), and set 
	\(\varepsilon\coloneq r'-r>0\). 
	
	By Definition~\ref{def_complete_irrationality}, there exist a sequence 
	\((\tau_j,\xi_j)\in\mathbb Z\times\mathbb N\), with
	\(\xi_j\to\infty\), such that
	\[
	0<|\tau_j-a\xi_j|
	<
	\xi_j^{1-r'}
	=
	\xi_j^{1-r}\xi_j^{-\varepsilon}.
	\]
	
	Passing to a subsequence, we may assume
	\(\xi_j^\varepsilon\geq j\), and hence
	\begin{equation}\label{eq:r-below-mu-approximating-sequence}
		0<|\tau_j-a\xi_j|
		<
		\frac1j\,\xi_j^{1-r}.
	\end{equation}
	
	Since \(r'>1\),
	\[
	\langle\tau_j\rangle\asymp_a\langle\xi_j\rangle.
	\]
	
	Define \(u\in\mathcal D'(\mathbb T^2)\) by
	\[
	\widehat u(\tau_j,\xi_j)
	=
	\langle\tau_j\rangle^{-k_1}
	\langle\xi_j\rangle^{-k_2-1+r},
	\]
	and set all other Fourier coefficients equal to zero. These
	coefficients have polynomial growth.
	
	Set \(f=Lu\). By
	\eqref{eq:r-below-mu-approximating-sequence},
	\[
	\|f\|_{H_{\operatorname{mix}}^{(k_1,k_2)}}^2
	=
	\sum_j
	|\tau_j-a\xi_j|^2
	\langle\xi_j\rangle^{-2+2r}
	\lesssim
	\sum_j\frac1{j^2}
	<\infty.
	\]
	
	Thus
	\(f\in H_{\operatorname{mix}}^{(k_1,k_2)}(\mathbb T^2)\).
	
	On the other hand, for any fixed \(\theta\in\mathbb R\),
	\[
	\langle\tau_j\rangle^{2(k_1+\theta)}
	\langle\xi_j\rangle^{2(k_2+1-r-\theta)}
	|\widehat u(\tau_j,\xi_j)|^2
	=
	\left(
	\frac{\langle\tau_j\rangle}{\langle\xi_j\rangle}
	\right)^{2\theta}
	\asymp_{a,\theta}1.
	\]
	
	Hence the corresponding mixed Sobolev series diverges, so
	\(u\notin H_{\operatorname{mix}}^{(k_1+\theta,k_2+1-r-\theta)}(\mathbb T^2)\).
	Since \(\theta\in\mathbb R\) was arbitrary, the result follows.
\end{proof}

The mixed smoothness estimates above do not in general recover the
optimal isotropic regularity via
Proposition~\ref{prop:isotropic-mixed-embeddings}. For comparison, we
recall the isotropic result from \cite[Theorem~4.10]{KowKir2026-jfa}.

\begin{proposition}\label{prop:isotropic-regularity-constant}
	Let \(s\in\mathbb R\), and suppose that
	\(u\in\mathcal D'(\mathbb T^2)\) is a distributional solution of
	\[
	Lu=f,
	\qquad
	f\in H^s(\mathbb T^2).
	\]
	Then the following statements hold.
	\begin{enumerate}
		\item[(i)]
		If \(b\neq0\), then
		\[
		u\in H^{s+1}(\mathbb T^2).
		\]
		
		\item[(ii)]
		If \(b=0\) and
		\(a\in\mathbb R\setminus\mathbb Q\) has finite irrationality
		measure, then, for every \(r>\mu(a)\),
		\[
		u\in H^{s+1-r}(\mathbb T^2).
		\]
	\end{enumerate}
\end{proposition}

%=================================================
\subsection{Rational and Liouville obstructions} \
%=================================================

The next results show that the arithmetic assumptions in
Theorem~\ref{theo_regularity_cte} are essential. In the rational
case, the symbol vanishes at infinitely many frequencies, whereas in
the Liouville case it becomes smaller than any prescribed polynomial
rate along suitable sequences of integer frequencies.

\begin{proposition}\label{prop:rational-no-mixed-regularity}
	Let \(a\in\mathbb Q\), and set \(L=\partial_t-a\,\partial_x\).
	Then, for every \(k_1,k_2\in\mathbb R\), there exists a nonzero
	distribution \(u\in\mathcal D'(\mathbb T^2)\) such that
	\[
	Lu=0
	\qquad\text{and}\qquad
	u\notin
	H_{\operatorname{mix}}^{(k_1,k_2)}(\mathbb T^2).
	\]
\end{proposition}

\begin{proof}
	Write \(a=p/q\), with \(p\in\mathbb Z, q\in\mathbb N\), and \(\gcd(p,q)=1\).
	Then every frequency \((mp,mq)\), \(m\in\mathbb N\), is resonant.
	Define
	\[
	\widehat u(mp,mq)
	=
	\langle mp\rangle^{-k_1}\langle mq\rangle^{-k_2},
	\qquad m\in\mathbb N,
	\]
	and set all other Fourier coefficients equal to zero. These
	coefficients have polynomial growth, so they define
	\(u\in\mathcal D'(\mathbb T^2)\), and the resonance condition gives
	\(Lu=0\). On the other hand,
	\[
	\sum_{(\tau,\xi)\in\mathbb Z^2}
	\langle\tau\rangle^{2k_1}
	\langle\xi\rangle^{2k_2}
	|\widehat u(\tau,\xi)|^2
	=
	\sum_{m\in\mathbb N}1
	=
	\infty,
	\]
	so
	\(u\notin
	H_{\operatorname{mix}}^{(k_1,k_2)}(\mathbb T^2)\).
\end{proof}

\begin{proposition}\label{prop:liouville-no-mixed-regularity}
	Let \(a\in\mathbb R\setminus\mathbb Q\) be a Liouville number, and
	set \(L=\partial_t-a\,\partial_x\).
	Then, for every \(k_1,k_2\in\mathbb R\), there exist
	\(f\in C^\infty(\mathbb T^2)\) and
	\(u\in\mathcal D'(\mathbb T^2)\) such that
	\[
	Lu=f
	\qquad\text{and}\qquad
	u\notin
	H_{\operatorname{mix}}^{(k_1,k_2)}(\mathbb T^2).
	\]
\end{proposition}

\begin{proof}
	Fix \(k_1,k_2\in\mathbb R\). Since \(a\) is Liouville, there exist
	distinct pairs \((\tau_j,\xi_j)\in\mathbb Z\times\mathbb N\), with
	\(\xi_j\to\infty\),  such that
	\begin{equation}\label{eq:Liouville-approximating-sequence}
		0<|\tau_j-a\xi_j|<\xi_j^{-j}.
	\end{equation}
	
	Since \(\tau_j/\xi_j\to a\), after discarding finitely many terms
	we may assume that
	\[
	\langle\tau_j\rangle\asymp_a\langle\xi_j\rangle\asymp\xi_j.
	\]
	
	Define
	\[
	\widehat u(\tau_j,\xi_j)
	=
	\langle\tau_j\rangle^{-k_1}
	\langle\xi_j\rangle^{-k_2},
	\]
	and set all other Fourier coefficients equal to zero. These
	coefficients have polynomial growth and hence define
	\(u\in\mathcal D'(\mathbb T^2)\).
	
	Set \(f=Lu\). If \(M=|k_1|+|k_2|\), then for every
	\(N\in\mathbb N_0\),
	\[
	\langle(\tau_j,\xi_j)\rangle^N
	|\widehat f(\tau_j,\xi_j)|
	\lesssim_{a,k_1,k_2,N}
	\xi_j^{N+M-j}.
	\]
	
	Thus the Fourier coefficients of \(f\) are rapidly decreasing, and
	hence \(f\in C^\infty(\mathbb T^2)\).
	
	Finally,
	\[
	\sum_{(\tau,\xi)\in\mathbb Z^2}
	\langle\tau\rangle^{2k_1}
	\langle\xi\rangle^{2k_2}
	|\widehat u(\tau,\xi)|^2
	=
	\sum_{j\in\mathbb N}1
	=
	\infty.
	\]
	Therefore
	\(u\notin
	H_{\operatorname{mix}}^{(k_1,k_2)}(\mathbb T^2)\).
\end{proof}

The two propositions show that, for rational or Liouville \(a\), no
prescribed finite loss in the two variables yields a universal mixed
smoothness estimate. More precisely, for every
\(k_1,k_2\in\mathbb R\) and \(r_1,r_2\geq0\), there exist
\[
f\in
H_{\operatorname{mix}}^{(k_1,k_2)}(\mathbb T^2),
\qquad
u\in\mathcal D'(\mathbb T^2),
\]
with \(Lu=f\), such that
\[
u\notin
H_{\operatorname{mix}}^{
	(k_1-r_1,k_2-r_2)}(\mathbb T^2).
\]

%==========================================
%==========================================
\section{Real-Valued Variable-Coefficient Vector Fields} 
\label{sec:variable-real}
%==========================================
%==========================================

We now consider the variable-coefficient vector field
\[
L\coloneq\partial_t-a(t)\partial_x,
\qquad
a\in C^\infty(\mathbb T;\mathbb R),
\]
and denote the average of \(a\) by
\[
a_0\coloneq
\frac{1}{2\pi}\int_0^{2\pi}a(t)\,dt.
\]

The small-divisor structure is determined by \(a_0\), while
differentiation in \(t\) produces additional powers of the Fourier
frequency \(\xi\).

Throughout this section, we assume that
\(a_0\in\mathbb R\setminus\mathbb Q\) has finite irrationality
measure. Then every nonzero Fourier mode is nonresonant, and
Lemma~\ref{lemma_exp_diof} gives, for every \(r>\mu(a_0)\),
\begin{equation}\label{eq:variable-exponential-Diophantine}
	|1-e^{\pm2\pi i\xi a_0}|
	\gtrsim_{a_0,r}
	\langle\xi\rangle^{1-r},
	\qquad
	\xi\in\mathbb Z\setminus\{0\}.
\end{equation}

%=================================================
\subsection{Mixed regularity and solvability} \
%=================================================

We first establish Sobolev estimates for the periodic equations obtained
after taking the partial Fourier transform in \(x\).

\begin{lemma}\label{lem:variable-real-mode-Sobolev-estimate}
	Let \(a\in C^\infty(\mathbb T;\mathbb R)\), and suppose that
	\(a_0\) is irrational and has finite irrationality measure. Let
	\(r>\mu(a_0)\), \(s\in\mathbb R\), and \(0\leq\theta\leq1\).
	For every \(\xi\in\mathbb Z\setminus\{0\}\) and every
	\(h\in H^s(\mathbb T)\), the equation
	\begin{equation*}\label{eq:variable-real-mode-operator}
		\partial_t v(t)-i\xi a(t)v(t)=h(t)
	\end{equation*}
	has a unique distributional solution
	\(v\in\mathcal D'(\mathbb T)\). Moreover,
	\(v\in H^{s+\theta}(\mathbb T)\) and
	\begin{equation}\label{eq:variable-real-mode-Sobolev-estimate}
		\|v\|_{H^{s+\theta}(\mathbb T)}
		\lesssim_{a,s,r,\theta}
		\langle\xi\rangle^{r-1+|s|+\theta}
		\|h\|_{H^s(\mathbb T)}.
	\end{equation}
\end{lemma}

\begin{proof}
	Fix \(\xi\in\mathbb Z\setminus\{0\}\). For
	\(h\in C^\infty(\mathbb T)\), let \(v=T_\xi h\) be the unique
	periodic solution given by
	\eqref{eq:general-variable-solution-backward}. Since \(a\) is
	real-valued, Minkowski's integral inequality and
	\eqref{eq:variable-exponential-Diophantine} give
	\begin{equation}\label{eq:variable-real-mode-L2-estimate}
		\|T_\xi h\|_{L^2(\mathbb T)}
		\lesssim_{a_0,r}
		\langle\xi\rangle^{r-1}
		\|h\|_{L^2(\mathbb T)}.
	\end{equation}
	
	For \(m\in\mathbb N_0\), the equation
	\[
	\partial_t v=h+i\xi a(t)v,
	\]
	 the alternative characterization of $H^m(\T)$ in terms of $L^2$-norms of derivatives, and Leibniz's  rule yield
	\[
	\|v\|_{H^{m+1}(\mathbb T)}
	\lesssim_{a,m}
	\|h\|_{H^m(\mathbb T)}
	+
	\langle\xi\rangle
	\|v\|_{H^m(\mathbb T)}.
	\]
    
	Starting from
	\eqref{eq:variable-real-mode-L2-estimate} and arguing by induction,
	we obtain
	\begin{equation}\label{eq:variable-real-mode-gain-integer-order}
		\|T_\xi h\|_{H^{m+1}(\mathbb T)}
		\lesssim_{a,m,r}
		\langle\xi\rangle^{r+m}
		\|h\|_{H^m(\mathbb T)}.
	\end{equation}
	
In particular, for \(m\geq1\), applying
\eqref{eq:variable-real-mode-gain-integer-order} with \(m-1\) and using
the continuous embedding
\(H^m(\mathbb T)\hookrightarrow H^{m-1}(\mathbb T)\) we obtain
\[
\|T_\xi h\|_{H^m(\mathbb T)}
\lesssim_{a,m,r}
\langle\xi\rangle^{r-1+m}
\|h\|_{H^m(\mathbb T)}.
\]

For \(m=0\), the same estimate follows directly from
\eqref{eq:variable-real-mode-L2-estimate}. Hence, for every
\(m\in\mathbb N_0\),
\begin{equation}\label{eq:variable-real-mode-same-integer-order}
	\|T_\xi h\|_{H^m(\mathbb T)}
	\lesssim_{a,m,r}
	\langle\xi\rangle^{r-1+m}
	\|h\|_{H^m(\mathbb T)}.
\end{equation}
	
	Interpolation between consecutive integer orders yields, for every
	\(s\geq0\),
	\begin{align}
		\|T_\xi h\|_{H^s(\mathbb T)}
		&\lesssim_{a,s,r}
		\langle\xi\rangle^{r-1+s}
		\|h\|_{H^s(\mathbb T)},                                     
		\label{eq:variable-real-mode-same-positive-order}
		\\
		\|T_\xi h\|_{H^{s+1}(\mathbb T)}
		&\lesssim_{a,s,r}
		\langle\xi\rangle^{r+s}
		\|h\|_{H^s(\mathbb T)}.
		\label{eq:variable-real-mode-gain-positive-order}
	\end{align}
	
	Interpolating once more, now with the domain \(H^s(\mathbb T)\)
	fixed, proves
	\eqref{eq:variable-real-mode-Sobolev-estimate} for \(s\geq0\).
	
	We now consider negative orders. Let \(s=-\sigma\), with
	\(\sigma>0\), and set
	\[
	P_{\xi,a}\coloneq\partial_t-i\xi a(t).
	\]
	
	Since \(a\) is real-valued, the \(L^2\)-formal adjoint of
	\(P_{\xi,a}\) is
	\[
	P_{\xi,a}^*
	=
	-\partial_t+i\xi a(t)
	=
	-P_{\xi,a}.
	\]
	
	Consequently, on smooth periodic functions, \(T_\xi^*=-T_\xi\).
	
	Using the duality between \(H^{-\sigma}(\mathbb T)\) and
	\(H^\sigma(\mathbb T)\), together with
	\eqref{eq:variable-real-mode-same-positive-order}, we obtain
	\begin{align*}
		\|T_\xi h\|_{H^{-\sigma}(\mathbb T)}
		&=
		\sup_{\varphi\in H^\sigma(\mathbb T)\setminus\{0\}}
		\frac{|\langle T_\xi h,\varphi\rangle|}
		{\|\varphi\|_{H^\sigma(\mathbb T)}}
		=
		\sup_{\varphi\neq0}
		\frac{
			|\langle h,T_\xi^*\varphi\rangle|
		}{
			\|\varphi\|_{H^\sigma(\mathbb T)}
		}
		\\
		&\leq
		\|h\|_{H^{-\sigma}(\mathbb T)}
		\sup_{\varphi\neq0}
		\frac{
			\|T_\xi\varphi\|_{H^\sigma(\mathbb T)}}
		{\|\varphi\|_{H^\sigma(\mathbb T)}}
		\lesssim_{a,\sigma,r}
		\langle\xi\rangle^{r-1+\sigma}
		\|h\|_{H^{-\sigma}(\mathbb T)}.
	\end{align*}
	Here \(\langle\cdot,\cdot\rangle\) denotes the duality pairing
	extending the \(L^2\)-pairing. Thus, for every \(s<0\),
	\begin{equation}\label{eq:variable-real-mode-same-negative-order}
		\|T_\xi h\|_{H^s(\mathbb T)}
		\lesssim_{a,s,r}
		\langle\xi\rangle^{r-1+|s|}
		\|h\|_{H^s(\mathbb T)}.
	\end{equation}
	
	Notice that the same-order estimate is essential in this duality
	argument: the estimate of \(T_\xi\) on \(H^\sigma(\mathbb T)\)
	transfers directly to the corresponding estimate on
	\(H^{-\sigma}(\mathbb T)\).
	
	Using again
	\[
	\partial_t v=h+i\xi a(t)v
	\]
	and the boundedness of multiplication by \(a\) on
	\(H^s(\mathbb T)\), we obtain
	\begin{align*}
		\|v\|_{H^{s+1}(\mathbb T)}
		&\lesssim_s
		\|v\|_{H^s(\mathbb T)}
		+
		\|\partial_t v\|_{H^s(\mathbb T)}
		\\
		&\lesssim_{a,s}
		\|h\|_{H^s(\mathbb T)}
		+
		\langle\xi\rangle
		\|v\|_{H^s(\mathbb T)}
		\\
		&\lesssim_{a,s,r}
		\langle\xi\rangle^{r+|s|}
		\|h\|_{H^s(\mathbb T)}.
	\end{align*}
	
	Hence
	\begin{equation}\label{eq:variable-real-mode-gain-negative-order}
		\|T_\xi h\|_{H^{s+1}(\mathbb T)}
		\lesssim_{a,s,r}
		\langle\xi\rangle^{r+|s|}
		\|h\|_{H^s(\mathbb T)}.
	\end{equation}
	
	Interpolating, with the domain \(H^s(\mathbb T)\) fixed, between
	\eqref{eq:variable-real-mode-same-negative-order} and
	\eqref{eq:variable-real-mode-gain-negative-order} proves
	\eqref{eq:variable-real-mode-Sobolev-estimate} for \(s<0\).
	
	Finally, since \(C^\infty(\mathbb T)\) is dense in
	\(H^s(\mathbb T)\), the estimates above extend \(T_\xi\) uniquely
	from smooth functions to the corresponding Sobolev spaces. More
	precisely, the estimate with \(\theta=1\) gives a bounded extension
	\[
	T_\xi:H^s(\mathbb T)\longrightarrow H^{s+1}(\mathbb T).
	\]
	Since \(P_{\xi,a}:H^{s+1}(\mathbb T)\longrightarrow H^s(\mathbb T)\)
	is continuous, the identity
	\[
	P_{\xi,a}T_\xi h=h,
	\]
	initially valid for smooth \(h\), extends by density to every
	\(h\in H^s(\mathbb T)\). Thus \(T_\xi h\) is a distributional
	solution of the equation. Uniqueness follows from the homogeneous
	equation and the fact that
	\(e^{2\pi i\xi a_0}\neq1\).
\end{proof}

\begin{theorem}\label{theo_regularity_real}
	Let \(a\in C^\infty(\mathbb T;\mathbb R)\), let
	\(k_1,k_2\in\mathbb R\), and suppose that
	\(u\in\mathcal D'(\mathbb T^2)\) satisfies 
	\(Lu=f \in 
	H_{\operatorname{mix}}^{(k_1,k_2)}(\mathbb T^2)\). 
	If \(a_0\in\mathbb R\setminus\mathbb Q\) has finite
	irrationality measure. Then, for every \(r>\mu(a_0)\), every
	\(q\leq k_1\), and every \(0\leq\theta\leq1\),
	\begin{equation}\label{eq:variable-real-intermediate-regularity}
		u\in
		H_{\operatorname{mix}}^{
			(q+\theta,\,
			k_2+1-r-|q|-\theta)}
		(\mathbb T^2).
	\end{equation}
	More precisely,
	\begin{equation}\label{eq:variable-real-quantitative-estimate}
		\left\|
		u-\widehat u(0,0)
		\right\|_{
			H_{\operatorname{mix}}^{
				(q+\theta,\,
				k_2+1-r-|q|-\theta)}
			(\mathbb T^2)}
		\lesssim_{a,k_1,q,r,\theta}
		\|f\|_{
			H_{\operatorname{mix}}^{(k_1,k_2)}
			(\mathbb T^2)}.
	\end{equation}
	
	If \(k_1\geq0\), then, for every \(r>\mu(a_0)\) and every
	\(0\leq\sigma\leq k_1+1\),
	\begin{equation}\label{eq:variable-real-nonnegative-family}
		u\in
		H_{\operatorname{mix}}^{
			(\sigma,k_2+1-r-\sigma)}
		(\mathbb T^2).
	\end{equation}
\end{theorem}

\begin{proof}
	Set \(v\coloneq u-\widehat u(0,0)\). Then
	\(\widehat v(0,0)=0\), \(Lv=f\), and
	\(\widehat f(0,0)=0\).
	
	Fix \(r>\mu(a_0)\), \(q\leq k_1\), and
	\(0\leq\theta\leq1\). For \(\xi\neq0\),
	Lemma~\ref{lem:variable-real-mode-Sobolev-estimate} and the
	embedding \(H^{k_1}(\mathbb T)\hookrightarrow H^q(\mathbb T)\)
	give
	\[
	\|\widehat v(\cdot,\xi)\|_{H^{q+\theta}(\mathbb T)}
	\lesssim_{a,k_1,q,r,\theta}
	\langle\xi\rangle^{r-1+|q|+\theta}
	\|\widehat f(\cdot,\xi)\|_{H^{k_1}(\mathbb T)}.
	\]
	
	Therefore
	\begin{equation}\label{eq:variable-real-nonzero-mode-estimate}
		\langle\xi\rangle^{k_2+1-r-|q|-\theta}
		\|\widehat v(\cdot,\xi)\|_{H^{q+\theta}(\mathbb T)}
		\lesssim_{a,k_1,q,r,\theta}
		\langle\xi\rangle^{k_2}
		\|\widehat f(\cdot,\xi)\|_{H^{k_1}(\mathbb T)}.
	\end{equation}
	
	For \(\xi=0\),
	\[
	\partial_t\widehat v(t,0)=\widehat f(t,0),
	\qquad
	\widehat v(0,0)=0.
	\]
	Hence
	\[
	\|\widehat v(\cdot,0)\|_{H^{q+\theta}(\mathbb T)}
	\lesssim
	\|\widehat f(\cdot,0)\|_{H^{q+\theta-1}(\mathbb T)}
	\lesssim
	\|\widehat f(\cdot,0)\|_{H^{k_1}(\mathbb T)},
	\]
	since \(q+\theta-1\leq q\leq k_1\).
	
	Squaring \eqref{eq:variable-real-nonzero-mode-estimate}, summing
	over \(\xi\neq0\), and including the zero mode, we obtain from
	\eqref{eq:mixed-partial-characterization}:
	\[
	\|v\|_{
		H_{\operatorname{mix}}^{
			(q+\theta,\,
			k_2+1-r-|q|-\theta)}
		(\mathbb T^2)}
	\lesssim_{a,k_1,q,r,\theta}
	\|f\|_{
		H_{\operatorname{mix}}^{(k_1,k_2)}
		(\mathbb T^2)}.
	\]
	This proves \eqref{eq:variable-real-quantitative-estimate}, and
	\eqref{eq:variable-real-intermediate-regularity} follows since
	\(u=v+\widehat u(0,0)\).
	
	Finally, suppose that \(k_1\geq0\) and
	\(0\leq\sigma\leq k_1+1\). Set
	\[
	q\coloneq\min\{\sigma,k_1\},
	\qquad
	\theta\coloneq\sigma-q.
	\]
	Then \(0\leq q\leq k_1\), \(0\leq\theta\leq1\),
	\(q+\theta=\sigma\), and \(|q|=q\). Thus
	\eqref{eq:variable-real-intermediate-regularity} gives
	\[
	u\in
	H_{\operatorname{mix}}^{
		(\sigma,k_2+1-r-\sigma)}
	(\mathbb T^2),
	\]
	which proves \eqref{eq:variable-real-nonnegative-family}.
\end{proof}
	
	\begin{corollary}\label{coro_existence_real}
		Let \(k_1,k_2\in\mathbb R\), and let
		\(f\in
		H_{\operatorname{mix}}^{(k_1,k_2)}(\mathbb T^2)\)
		satisfy \(\widehat f(0,0)=0\).
		Suppose that
		\(a_0\in\mathbb R\setminus\mathbb Q\) has finite irrationality
		measure. Then there exists a unique distributional solution
		\(u\in\mathcal D'(\mathbb T^2)\) of \(Lu=f\) satisfying
		\(\widehat u(0,0)=0\).
		
		Moreover, for every \(r>\mu(a_0)\), every \(q\leq k_1\), and
		every \(0\leq\theta\leq1\), this solution satisfies
		\eqref{eq:variable-real-intermediate-regularity} and
		\eqref{eq:variable-real-quantitative-estimate}.
	\end{corollary}
	
	\begin{proof}
		For \(\xi\neq0\), let \(u_\xi\) be the unique periodic solution of
		\(\partial_tu_\xi-i\xi a(t)u_\xi=\widehat f(\cdot,\xi)\) given by
		Lemma~\ref{lem:variable-real-mode-Sobolev-estimate}. Fix
		\(r_0>\mu(a_0)\). Applying the lemma with \(s=k_1\) and
		\(\theta=0\), we obtain
		\[
		\|u_\xi\|_{H^{k_1}(\mathbb T)}
		\lesssim_{a,k_1,r_0}
		\langle\xi\rangle^{r_0-1+|k_1|}
		\|\widehat f(\cdot,\xi)\|_{H^{k_1}(\mathbb T)}.
		\]
		
		For \(\xi=0\), define \(u_0\) by
		\[
		\widehat{u_0}(\tau)
		=
		\begin{cases}
			\dfrac{\widehat f(\tau,0)}{i\tau},
			&\tau\neq0,\\[0.6em]
			0,
			&\tau=0.
		\end{cases}
		\]
		
		The condition \(\widehat f(0,0)=0\) ensures that
		\(\partial_tu_0=\widehat f(\cdot,0)\), and
		\[
		\|u_0\|_{H^{k_1+1}(\mathbb T)}
		\lesssim_{k_1}
		\|\widehat f(\cdot,0)\|_{H^{k_1}(\mathbb T)}.
		\]
		
		Defining
		\(\widehat u(\cdot,\xi)=u_\xi\), the preceding estimates and
		\eqref{eq:mixed-partial-characterization} give \(Lu=f\), 
		\(\widehat u(0,0)=0\), and 
		\[
		u\in
		H_{\operatorname{mix}}^{
			(k_1,k_2+1-r_0-|k_1|)}(\mathbb T^2)
		\subset\mathcal D'(\mathbb T^2).
		\]
		
		Theorem~\ref{theo_regularity_real} now gives
		\eqref{eq:variable-real-intermediate-regularity} and
		\eqref{eq:variable-real-quantitative-estimate} for every
		\(r>\mu(a_0)\).
		
		Finally, if \(w\) is the difference of two such solutions, then
		\(Lw=0\). For every \(\xi\neq0\), uniqueness for the corresponding
		mode equation gives \(\widehat w(\cdot,\xi)=0\), while
		\(\partial_t\widehat w(t,0)=0\). Since
		\(\widehat w(0,0)=0\), the zero mode also vanishes. Hence \(w=0\).
	\end{proof}

\begin{remark}\label{rem:variable-constant-frontiers}
	The preceding theorem gives a useful geometric comparison between the
	mixed regularity obtained from the direct variable-coefficient estimates
	and that given by the averaged constant-coefficient field
	\(L_0\coloneq\partial_t-a_0\partial_x\).
	
	For the same data
	\(f\in H_{\operatorname{mix}}^{(k_1,k_2)}(\mathbb T^2)\),
	the constant-coefficient result yields the admissible region
	\[
	s_1\leq k_1+1,\qquad
	s_2\leq k_2+1,\qquad
	s_1+s_2<k_1+k_2+1-\mu(a_0).
	\]

	By contrast, taking the union of the direct variable-coefficient
	estimates over \(r> \mu(a_0)\) 	gives, when \(k_1\geq0\),
	\[
	s_1\leq k_1+1,\qquad
	s_2<k_2+1-\mu(a_0),\qquad
	s_1+s_2<k_2+1-\mu(a_0),
	\]
	while for \(k_1<0\) it gives
	\[
	s_1\leq k_1+1,\qquad
	s_2<k_2+1-\mu(a_0)-|k_1|,\qquad
	s_1+s_2<k_2+1-\mu(a_0)-2|k_1|.
	\]
	
	Thus the region obtained from the direct variable-coefficient analysis
	is contained in the corresponding constant-coefficient region for
	\(L_0\).  Figure~\ref{fig:variable-constant-frontiers} displays the
	two boundaries for a visual comparison.

	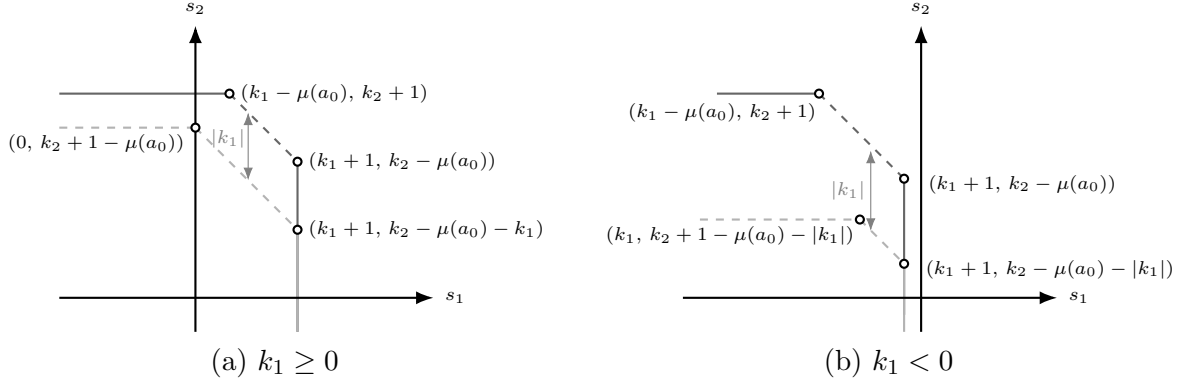
\begin{figure}[!htb]
		\centering

		% Dark gray: averaged constant-coefficient field L_0
		% Light gray: direct variable-coefficient estimates
		\tikzset{
			constfront/.style={thick, black!60},
			varfront/.style={thick, black!30},
			openpoint/.style={
				circle, draw=black, fill=white, thick,
				inner sep=0.5pt, minimum size=3pt
			},
			gaparrow/.style={<->, thin, black!50}
		}

		\begin{minipage}[t]{0.48\textwidth}
			\centering
			\begin{tikzpicture}[
				x=0.45cm,
				y=0.45cm,
				every node/.style={font=\fontsize{7}{8.5}\selectfont},
				>={Latex}
				]
				%-------------------------------------------------
				% Panel (a): k_1 >= 0
				%-------------------------------------------------

				% Averaged constant-coefficient frontier
				\coordinate (CTopLeft)     at (-4,6);
				\coordinate (CTopOpen)     at (1,6);
				\coordinate (CBottomOpen)  at (3,4);
				\coordinate (CVerticalEnd) at (3,-1);

				% Variable-coefficient frontier
				\coordinate (VTopLeft)     at (-4,5);
				\coordinate (VTopOpen)     at (0,5);
				\coordinate (VBottomOpen)  at (3,2);
				\coordinate (VVerticalEnd) at (3,-1);

				% Constant-coefficient frontier:
				% horizontal and vertical parts are included;
				% the slanted endpoint frontier r=mu_0 is not included in general.
				\draw[constfront] (CTopLeft) -- (CTopOpen);
				\draw[constfront, dashed] (CTopOpen) -- (CBottomOpen);
				\draw[constfront] (CBottomOpen) -- (CVerticalEnd);

				% Variable-coefficient frontier:
				% the horizontal and slanted upper frontier corresponds to
				% the limiting endpoint r=mu_0, hence is dashed.
				\draw[varfront, dashed] (VTopLeft) -- (VTopOpen);
				\draw[varfront, dashed] (VTopOpen) -- (VBottomOpen);
				\draw[varfront] (VBottomOpen) -- (VVerticalEnd);

				% Axes
				\draw[->, thick, black] (-4,0) -- (7,0)
				node[right] {$s_1$};
				\draw[->, thick, black] (0,-1) -- (0,8)
				node[above] {$s_2$};

				% Open points
				\node[openpoint] at (CTopOpen) {};
				\node[openpoint] at (CBottomOpen) {};
				\node[openpoint] at (VTopOpen) {};
				\node[openpoint] at (VBottomOpen) {};

				% Labels: averaged constant coefficient
				\node[black, anchor=west] at (CTopOpen)
				{$(k_1-\mu(a_0),\,k_2+1)$};
				\node[black, anchor=west] at (CBottomOpen)
				{$(k_1+1,\,k_2-\mu(a_0))$};

				% Labels: variable coefficient
				\node[black, anchor=east, yshift=-6pt] at (VTopOpen)
				{$(0,\,k_2+1-\mu(a_0))$};
				\node[black, anchor=west] at (VBottomOpen)
				{$(k_1+1,\,k_2-\mu(a_0)-k_1)$};

				% Vertical separation between the two oblique supporting lines
				\draw[gaparrow] (1.55,5.45) -- (1.55,3.45)
				node[midway, left, xshift=3pt, yshift=3pt] {$|k_1|$};
			\end{tikzpicture}
			{\small (a) \(k_1\geq0\)}
		\end{minipage}
		\hfill
		\begin{minipage}[t]{0.5\textwidth}
			\centering
			\begin{tikzpicture}[
				x=0.45cm,
				y=0.45cm,
				every node/.style={font=\fontsize{7}{8.5}\selectfont},
				>={Latex}
				]
				%-------------------------------------------------
				% Panel (b): k_1 < 0
				% The drawing is schematic; the relative position of
				% k_1+1 and the s_2-axis depends on the value of k_1.
				%-------------------------------------------------

				% Averaged constant-coefficient frontier
				\coordinate (CTopLeftB)     at (-6,6);
				\coordinate (CTopOpenB)     at (-3,6);
				\coordinate (CBottomOpenB)  at (-0.5,3.5);
				\coordinate (CVerticalEndB) at (-0.5,-0.5);

				% Variable-coefficient frontier
				\coordinate (VTopLeftB)     at (-6.5,2.3);
				\coordinate (VTopOpenB)     at (-1.8,2.3);
				\coordinate (VBottomOpenB)  at (-0.5,1);
				\coordinate (VVerticalEndB) at (-0.5,-1);

				% Constant-coefficient frontier
				\draw[constfront] (CTopLeftB) -- (CTopOpenB);
				\draw[constfront, dashed] (CTopOpenB) -- (CBottomOpenB);
				\draw[constfront] (CBottomOpenB) -- (CVerticalEndB);

				% Variable-coefficient frontier
				\draw[varfront, dashed] (VTopLeftB) -- (VTopOpenB);
				\draw[varfront, dashed] (VTopOpenB) -- (VBottomOpenB);
				\draw[varfront] (VBottomOpenB) -- (VVerticalEndB);

				% Axes
				\draw[->, thick, black] (-7,0) -- (4,0)
				node[right] {$s_1$};
				\draw[->, thick, black] (0,-1) -- (0,8)
				node[above] {$s_2$};

				% Open points
				\node[openpoint] at (CTopOpenB) {};
				\node[openpoint] at (CBottomOpenB) {};
				\node[openpoint] at (VTopOpenB) {};
				\node[openpoint] at (VBottomOpenB) {};

				% Labels: averaged constant coefficient
				\node[black, anchor=east, yshift=-7pt, xshift=3pt] at (CTopOpenB)
				{$(k_1-\mu(a_0),\,k_2+1)$};
				\node[black, anchor=west, xshift=1pt, yshift=-2pt, xshift=4pt]
				at (CBottomOpenB)
				{$(k_1+1,\,k_2-\mu(a_0))$};

				% Labels: variable coefficient
				\node[black, anchor=east, yshift=-6pt, xshift=2pt] at (VTopOpenB)
				{$(k_1,\,k_2+1-\mu(a_0)-|k_1|)$};
				\node[black, anchor=west, xshift=-2pt, yshift=-2pt, xshift=6pt]
				at (VBottomOpenB)
				{$(k_1+1,\,k_2-\mu(a_0)-|k_1|)$};

				% Vertical separation between the two oblique supporting lines
				\draw[gaparrow] (-1.48,4.33) -- (-1.48,1.96)
				node[midway, left, xshift=2pt] {$|k_1|$};
			\end{tikzpicture}
			{\small (b) \(k_1<0\)}
		\end{minipage}

		\caption{\footnotesize
			Comparison of the admissible mixed-regularity regions for
			\(Lu=f\), with
			\(f\in H_{\operatorname{mix}}^{(k_1,k_2)}(\mathbb T^2)\).
			In each panel, the admissible region is the lower-left region
			bounded by the graphed lines.
			The dark-gray boundary corresponds to the averaged
			constant-coefficient field
			\(L_0=\partial_t-a_0\partial_x\), whereas the  light-gray lines correspond to the
		 boundary of the region obtained from the direct estimates for
			\(L=\partial_t-a(t)\partial_x\).
			The two regions have the same rightmost boundary
			\(s_1=k_1+1\), while the oblique supporting line for the
			variable-coefficient estimate  is parallel to the corresponding constant-coefficient line and differs from it
             vertically  by \(|k_1|\).
			Dashed portions correspond to the limiting exponent
		 \(r=\mu(a_0)\) and are not attained in general.
			The diagrams are schematic and not drawn to scale.}
		\label{fig:variable-constant-frontiers}
	\end{figure}

	The figure highlights three features of the comparison.  First, the
	maximal regularity in the \(t\)-direction is unchanged: both estimates
	have the same vertical boundary \(s_1=k_1+1\).  Second,  the maximal regularity in the $x$-direction of the variable-coefficient case lies below that of the averaged problem, represented as a pair of parallel horizontal lines, with a distance of \(\mu_0\) when \(k_1\geq0\) and of
	\(\mu(a_0)+|k_1|\) when \(k_1<0\).  Most importantly, the two oblique
	frontiers are parallel.  Their supporting lines are
	\[
	s_1+s_2=k_1+k_2+1-\mu(a_0)
	\]
	for the averaged constant-coefficient problem and
	\[
	s_1+s_2=
	\begin{cases}
		k_2+1-\mu(a_0), & k_1\geq0,\\[0.3em]
		k_2+1-\mu(a_0)-2|k_1|, & k_1<0,
	\end{cases}
	\]
	for the direct variable-coefficient estimates.  In either case, the
	second line is shifted downward by exactly \(|k_1|\).  Thus
	\(|k_1|\) measures precisely the additional loss in the mixed-scale
	information produced by the direct modal argument when compared with
	the averaged constant-coefficient model.

	This geometry can be read directly from
	\eqref{eq:variable-real-intermediate-regularity}.  For each fixed
	\(q\leq k_1\), the parameter \(0\leq\theta\leq1\) describes the
	unit segment
	\[
	(s_1,s_2)
	=
	\bigl(
	q+\theta,\,
	k_2+1-r-|q|-\theta
	\bigr),
	\]
	which lies on the line
	\[
	s_1+s_2=k_2+1-r+q-|q|.
	\]
	When \(q\geq0\), these segments lie on the common line
	\[
	s_1+s_2=k_2+1-r,
	\]
	whereas for \(q<0\) they are shifted downward to
	\[
	s_1+s_2=k_2+1-r-2|q|.
	\]
	The latter loss reflects the factor
	\(\langle\xi\rangle^{|q|}\) in the negative-order modal estimate.

	If the Diophantine lower bound is valid at \(r=\mu_0\), the
	corresponding limiting portions of the boundaries are included as well.
\end{remark}

	%=================================================
	\subsection{Isotropic consequences} \
	%=================================================
	
	We now specialize Theorem~\ref{theo_regularity_real} to isotropic
	data and determine the best isotropic regularity obtainable from the
	resulting direct mixed smoothness estimates.
	
	\begin{corollary}\label{coro_regularity_variable_real}
		Let \(k\geq0\), and suppose that
		\(u\in\mathcal D'(\mathbb T^2)\) is a solution of
		\[
		Lu=f,
		\qquad
		f\in H^k(\mathbb T^2).
		\]
		
		Suppose that \(a_0\in\mathbb R\setminus\mathbb Q\) has finite
		irrationality measure. Then,
		for every \(r>\mu(a_0)\),
		\begin{equation}\label{eq:mixed-consequence-isotropic-data}
			\begin{aligned}
				u&\in
				H_{\operatorname{mix}}^{
					(\sigma,k+1-r-\sigma)
				}(\mathbb T^2),
				&&0\leq\sigma\leq1,
				\\
				u&\in
				H_{\operatorname{mix}}^{
					(\sigma,k+2-r-2\sigma)
				}(\mathbb T^2),
				&&1\leq\sigma\leq k+1.
			\end{aligned}
		\end{equation}

		Consequently,
		\[
		u\in H^{\rho_{\operatorname{mix}}(k,r)}(\mathbb T^2),
		\]
		where
		\begin{equation}\label{eq:isotropic-variable-regularity-index}
			\rho_{\operatorname{mix}}(k,r)
			=
			\begin{cases}
				k+1-r,
				& k-r\leq-1,
				\\[0.4em]
				\dfrac{k+1-r}{2},
				& -1\leq k-r\leq1,
				\\[0.8em]
				\dfrac{k+2-r}{3},
				& k-r\geq1.
			\end{cases}
		\end{equation}
	\end{corollary}
	
	\begin{proof}
		Let \(r>\mu(a_0)\). For every \(0\leq p\leq k\),
		Proposition~\ref{prop:isotropic-mixed-embeddings} gives
		\(H^k(\mathbb T^2) \hookrightarrow
		H_{\operatorname{mix}}^{(p,k-p)}(\mathbb T^2)\). 
		Hence Theorem~\ref{theo_regularity_real}, applied with
		\(q\leq p\) and \(0\leq\theta\leq1\), yields
		\begin{equation}\label{eq:isotropic-data-general-mixed-pair}
			u\in
			H_{\operatorname{mix}}^{
				(q+\theta,\,
				k-p+1-r-|q|-\theta)}
			(\mathbb T^2).
		\end{equation}
		
		If \(q<0\), the pair in
		\eqref{eq:isotropic-data-general-mixed-pair} is componentwise
		dominated by the one obtained with \(q=0\), with \(p\) and
		\(\theta\) fixed. Thus it suffices to consider \(0\leq q\leq p\).
		
		Set \(\sigma=q+\theta\). The second index then becomes
		\(k+1-r-\sigma-p\).
		
		For fixed \(\sigma\), this is maximized by minimizing \(p\).
		Since \(p\geq q\) and \(0\leq\sigma-q\leq1\), the optimal choice is
		\[
		p=q=\max\{0,\sigma-1\}.
		\]
		
		Therefore
		\[
		u\in
		H_{\operatorname{mix}}^{
			(\sigma,k+1-r-\sigma)}(\mathbb T^2),
		\qquad
		0\leq\sigma\leq1,
		\]
		and
		\[
		u\in
		H_{\operatorname{mix}}^{
			(\sigma,k+2-r-2\sigma)}(\mathbb T^2),
		\qquad
		1\leq\sigma\leq k+1.
		\]
		
		This proves \eqref{eq:mixed-consequence-isotropic-data}.
		
		It remains to optimize the isotropic order. Since
		\(s_1=\sigma\geq0\), Proposition~
		\ref{prop:isotropic-mixed-embeddings} gives the isotropic index
		\(\min\{s_1,s_2\}\). Set \(m\coloneq k-r\). The two families lead
		to
		\[
		\rho_1(\sigma)
		=
		\min\{\sigma,m+1-\sigma\},
		\qquad
		0\leq\sigma\leq1,
		\]
		and
		\[
		\rho_2(\sigma)
		=
		\min\{\sigma,m+2-2\sigma\},
		\qquad
		1\leq\sigma\leq k+1.
		\]
		
		Elementary inspection gives
		\[
		\max\rho_1=
		\begin{cases}
			m+1, & m\leq-1,\\[0.3em]
			\dfrac{m+1}{2}, & -1\leq m\leq1,\\[0.6em]
			1, & m\geq1,
		\end{cases}
		\]
		whereas
		\[
		\max\rho_2=
		\begin{cases}
			m, & m\leq1,\\[0.3em]
			\dfrac{m+2}{3}, & m\geq1.
		\end{cases}
		\]
		Comparing these values and recalling that \(m=k-r\) yields
		\eqref{eq:isotropic-variable-regularity-index}.
	\end{proof}
	
	\begin{remark}
		The index \(\rho_{\operatorname{mix}}(k,r)\) is the largest
		isotropic order obtainable from the direct mixed smoothness
		estimates in \eqref{eq:mixed-consequence-isotropic-data} through
		Proposition~\ref{prop:isotropic-mixed-embeddings}. The normal-form
		conjugation developed in Section~\ref{sec:normal-form} yields the
		stronger isotropic conclusion
		\[
		u\in H^{k+1-r}(\mathbb T^2).
		\]
		Thus the direct mixed approach retains directional regularity
		information, but is not optimal after passing to the isotropic
		Sobolev scale.
	\end{remark}
	
	%=================================================
	\subsection{Classical regularity} \
	%=================================================
	
	We recall that \(L=\partial_t-a(t)\partial_x\), with
	\(a\in C^\infty(\mathbb T;\mathbb R)\), and denote the average of
	\(a\) by
	\[
	a_0\coloneq\frac{1}{2\pi}\int_0^{2\pi}a(t)\,dt.
	\]
	
	The preceding mixed smoothness estimates yield the following
	classical regularity consequence.
	
	\begin{corollary}\label{cor:variable-real-classical-regularity}
		Assume that
		\(a_0\in\mathbb R\setminus\mathbb Q\) has finite irrationality
		measure. Let \(k>\mu(a_0)\), and suppose that
		\(u\in\mathcal D'(\mathbb T^2)\) is a solution of \(Lu=f\), with
		\(f\in H^k(\mathbb T^2)\).
		Then \(u\in C^m(\mathbb T^2)\), where
		\[
		m
		\coloneq
		\left\lceil
		\frac{k-\mu(a_0)+\frac12}{3}
		\right\rceil-1.
		\]
	\end{corollary}
	
	\begin{proof}
		By the definition of \(m\),
		\begin{equation}\label{eq:classical-m-condition}
			3m-\frac12<k-\mu(a_0).
		\end{equation}
		
		We first consider \(m=0\). Choose
		\(\mu(a_0)<r<k\) and then \(\varepsilon>0\) sufficiently small
		so that \(\varepsilon<\min\left\{k-r, 1/2\right\}\).
		
		By the first family in
		\eqref{eq:mixed-consequence-isotropic-data}, with
		\(\sigma=\frac12+\varepsilon\), we have 
		\[
		u\in H_{\operatorname{mix}}^{
			\left(1/2+\varepsilon, k+1/2-r-\varepsilon \right)}
		(\mathbb T^2).
		\]
		
		Both indices are strictly larger than \(1/2\). Hence
		Corollary~\ref{cor:mixed-Cm-embedding} gives
		\(u\in C^0(\mathbb T^2)\).
		
		Now suppose that \(m\geq1\). By
		\eqref{eq:classical-m-condition}, we may choose
		\(r>\mu(a_0)\) sufficiently close to \(\mu(a_0)\), and then
		\(\varepsilon>0\) sufficiently small, so that
		\[
		3m-\frac12+2\varepsilon<k-r.
		\]
		
		Set \(\sigma\coloneq m+1/2+\varepsilon\). 
		Then \(1\leq\sigma\leq k+1\), and the second family in
		\eqref{eq:mixed-consequence-isotropic-data} gives
		\[
		u\in
		H_{\operatorname{mix}}^{
			\left(
			m+\frac12+\varepsilon,\,
			k+1-r-2m-2\varepsilon
			\right)}
		(\mathbb T^2).
		\]
		
		The choice of \(r\) and \(\varepsilon\) implies
		\(k+1-r-2m-2\varepsilon > m+1/2\). 
		Thus both mixed smoothness indices are strictly larger than
		\(m+\frac12\), and
		Corollary~\ref{cor:mixed-Cm-embedding} yields
		\(u\in C^m(\mathbb T^2)\).
	\end{proof}
	
	\begin{corollary}\label{cor:variable-real-classical-existence}
		Let \(k\geq0\), let \(f\in H^k(\mathbb T^2)\) satisfy
		\(\widehat f(0,0)=0\), and suppose that
		\(a_0\in\mathbb R\setminus\mathbb Q\) has finite irrationality
		measure. In particular,  when \(k\in\mathbb N_0\), the hypothesis
		\(f\in H^k(\mathbb T^2)\) is satisfied by every
		\(f\in C^k(\mathbb T^2)\).
		
		Then there exists a unique distributional solution
		\(u\in\mathcal D'(\mathbb T^2)\) of \(Lu=f\) satisfying
		\(\widehat u(0,0)=0\).
		
		If, in addition, \(k>\mu(a_0)\), then \(u\in C^m(\mathbb T^2)\), where
		\[
		m \coloneq
		\left\lceil
		\frac{k-\mu(a_0)+\frac12}{3}
		\right\rceil-1.
		\]
	\end{corollary}
	
	\begin{proof}
		Since \(k\geq0\),   
		\(H^k(\mathbb T^2) \hookrightarrow
		H_{\operatorname{mix}}^{(0,k)}(\mathbb T^2)\).
		Thus Corollary~\ref{coro_existence_real} gives the unique
		distributional solution satisfying \(\widehat u(0,0)=0\).
		If \(k>\mu(a_0)\), Corollary~
		\ref{cor:variable-real-classical-regularity} gives
		\(u\in C^m(\mathbb T^2)\).
	\end{proof}

	%==========================================
	%==========================================
	\section{Normal-Form Reduction: Mixed and Isotropic Regularity}
	\label{sec:normal-form}
	%==========================================
	%==========================================
	
	We now consider the real variable-coefficient vector field
	\[
	L\coloneq\partial_t-a(t)\partial_x,
	\qquad
	a\in C^\infty(\mathbb T;\mathbb R),
	\]
	from a different point of view. Let
	\[
	a_0\coloneq\frac{1}{2\pi}\int_0^{2\pi}a(t)\,dt,
	\ \text{ and } \
	L_0\coloneq\partial_t-a_0\partial_x.
	\]
	
	In the preceding section, regularity estimates were obtained directly
	from the periodic mode equations. Here we instead conjugate \(L\) to
	the constant-coefficient vector field \(L_0\).
	
	Such normal-form reductions are standard in the study of vector
	fields on tori; see, for instance,
	\cite{AriKirMed19_jmaa,DicGraYos2002_pems,Petr2011_tams,
		BerDatGon2017-jfaa}. Related procedures for vector fields on compact
	Lie groups and for other classes of operators can be found in
	\cite{KMR2020_bsm,KMR2021_jfa,Kiri_Kow_Wagn_ToroEsfera,
		AviGonKirMed19_jfaa,AviGraKir18_jam}.
	
	The conjugation will be used in two complementary ways. In the mixed
	smoothness scale, we quantify the loss in the \(x\)-direction produced
	by the conjugating operators and compare the resulting estimates with
	the direct method of the preceding section. In the isotropic Sobolev
	scale, the conjugation preserves the Sobolev order and allows us to
	transfer regularity and solvability results from \(L_0\) to \(L\).

	%============================================
	\subsection{The normal-form conjugation} \
	%============================================
	
	Set
	\begin{equation*}\label{eq:normal-form-phase}
		A(t) \coloneq \int_0^t\bigl(a(\sigma)-a_0\bigr)\,d\sigma.
	\end{equation*}
	
	Since \(a-a_0\) has zero mean, \(A\) is \(2\pi\)-periodic.
	
	\begin{proposition}\label{prop:normal-form-conjugation}
		For \(u\in\mathcal D'(\mathbb T^2)\), set
		\[
		(\Psi_{\pm A}u)(t,x)
		\coloneq
		u(t,x\mp A(t)),
		\]
		where the expression is understood as pullback by the corresponding
		smooth diffeomorphism. Then \(\Psi_A\) and \(\Psi_{-A}\) are mutually
		inverse continuous automorphisms of
		\(\mathcal D'(\mathbb T^2)\) and of
		\(C^\infty(\mathbb T^2)\).
		
		Moreover,
		\begin{equation}\label{eq:normal-form-intertwining}
			\Psi_A L=L_0\Psi_A,
			\qquad
			\Psi_{-A}L_0=L\Psi_{-A},
		\end{equation}
		and, for every \(\xi\in\mathbb Z\),
		\begin{equation}\label{eq:conjugation-partial-Fourier}
			\widehat{\Psi_{\pm A}u}(t,\xi)
			=
			e^{\mp i\xi A(t)}
			\widehat u(t,\xi).
		\end{equation}
	\end{proposition}
	
	\begin{proof}
		The maps
		\[
		\Phi_{\pm A}(t,x)\coloneq(t,x\mp A(t))
		\]
		are mutually inverse smooth diffeomorphisms of \(\mathbb T^2\).
		Hence their pullbacks are mutually inverse continuous
		automorphisms of \(\mathcal D'(\mathbb T^2)\) and
		\(C^\infty(\mathbb T^2)\). Their Jacobian determinants are equal
		to \(1\).
		
		For \(u\in C^\infty(\mathbb T^2)\), the chain rule and
		\(A'(t)=a(t)-a_0\) give
		\[
		L_0(\Psi_Au)
		=
		\Psi_A(Lu).
		\]
		Thus \(\Psi_A L=L_0\Psi_A\); the second identity in
		\eqref{eq:normal-form-intertwining} follows by applying
		\(\Psi_{-A}\). Both identities extend to distributions by
		continuity. Finally,
		\eqref{eq:conjugation-partial-Fourier} follows from the translation
		rule for Fourier coefficients in the \(x\)-variable.
	\end{proof}
	
	%============================================
	\subsection{Action on mixed smoothness spaces} \
	%============================================
	
	We next examine the action of \(\Psi_{\pm A}\) on Sobolev spaces of
	dominating mixed smoothness. By
	\eqref{eq:conjugation-partial-Fourier}, differentiation in \(t\) of
	the phase factor \(e^{\mp i\xi A(t)}\) produces powers of
	\(\langle\xi\rangle\). The following estimate quantifies the resulting
	loss in the \(x\)-regularity.
	
	\begin{lemma}\label{lem:normal-form-phase-multiplier}
		Let \(A\in C^\infty(\mathbb T;\mathbb R)\) and \(s\in\mathbb R\).
		Then, for every \(\xi\in\mathbb Z\) and every
		\(h\in H^s(\mathbb T)\),
		\begin{equation*}\label{eq:normal-form-phase-multiplier}
			\|e^{\pm i\xi A}h\|_{H^s(\mathbb T)}
			\lesssim_{A,s}
			\langle\xi\rangle^{|s|}
			\|h\|_{H^s(\mathbb T)}.
		\end{equation*}
	\end{lemma}
	
	\begin{proof}
		We first consider \(s=m\in\mathbb N_0\). For
		\(0\leq\ell\leq m\), repeated differentiation gives
		\[
		\left\|
		\partial_t^\ell e^{\pm i\xi A}
		\right\|_{L^\infty(\mathbb T)}
		\lesssim_{A,\ell}
		\langle\xi\rangle^\ell.
		\]
		
		Hence, by Leibniz's rule,
		\[
		\left\|
		\partial_t^j\bigl(e^{\pm i\xi A}h\bigr)
		\right\|_{L^2(\mathbb T)}
		\lesssim_{A,m}
		\langle\xi\rangle^m
		\|h\|_{H^m(\mathbb T)},
		\qquad
		0\leq j\leq m.
		\]
		
		Summing over \(j\) yields
		\begin{equation}\label{eq:phase-multiplier-integer}
			\|e^{\pm i\xi A}h\|_{H^m(\mathbb T)}
			\lesssim_{A,m}
			\langle\xi\rangle^m
			\|h\|_{H^m(\mathbb T)}.
		\end{equation}
		
		Now let \(s=m+\vartheta\geq0\), where
		\(m\in\mathbb N_0\) and \(0\leq\vartheta<1\).
		Interpolating \eqref{eq:phase-multiplier-integer} at orders
		\(m\) and \(m+1\) in the Sobolev scale, we obtain
		\[
		\|e^{\pm i\xi A}h\|_{H^s(\mathbb T)}
		\lesssim_{A,s}
		\langle\xi\rangle^{
			m(1-\vartheta)+(m+1)\vartheta}
		\|h\|_{H^s(\mathbb T)}
		=
		\langle\xi\rangle^s
		\|h\|_{H^s(\mathbb T)}.
		\]
		Finally, let \(s=-\sigma<0\). Since the \(L^2\)-adjoint of
		multiplication by \(e^{\pm i\xi A}\) is multiplication by
		\(e^{\mp i\xi A}\), the estimate already proved at order
		\(\sigma\) gives, by duality,
		\[
		\|e^{\pm i\xi A}h\|_{H^{-\sigma}(\mathbb T)}
		=
		\sup_{\varphi\in H^\sigma(\mathbb T)\setminus\{0\}}
		\frac{
			|\langle h,e^{\mp i\xi A}\varphi\rangle|
		}{
			\|\varphi\|_{H^\sigma(\mathbb T)}
		}
		\lesssim_{A,\sigma}
		\langle\xi\rangle^\sigma
		\|h\|_{H^{-\sigma}(\mathbb T)}.
		\]
		Since \(\sigma=|s|\), the result follows.
	\end{proof}
	
	\begin{proposition}\label{prop:conjugation-mixed-smoothness}
		Let \(s_1,s_2\in\mathbb R\). Then
		\[
		\Psi_{\pm A}\colon
		H_{\operatorname{mix}}^{(s_1,s_2)}(\mathbb T^2)
		\longrightarrow
		H_{\operatorname{mix}}^{
			(s_1,s_2-|s_1|)}(\mathbb T^2)
		\]
		is continuous. More precisely,
		\begin{equation}\label{eq:conjugation-mixed-smoothness}
			\|\Psi_{\pm A}u\|_{
				H_{\operatorname{mix}}^{
					(s_1,s_2-|s_1|)}(\mathbb T^2)}
			\lesssim_{A,s_1}
			\|u\|_{
				H_{\operatorname{mix}}^{(s_1,s_2)}
				(\mathbb T^2)}.
		\end{equation}
	\end{proposition}
	
	\begin{proof}
		By \eqref{eq:conjugation-partial-Fourier} and
		Lemma~\ref{lem:normal-form-phase-multiplier},
		\[
		\|\widehat{\Psi_{\pm A}u}(\cdot,\xi)\|_{H^{s_1}(\mathbb T)}
		\lesssim_{A,s_1}
		\langle\xi\rangle^{|s_1|}
		\|\widehat u(\cdot,\xi)\|_{H^{s_1}(\mathbb T)}.
		\]
		
		Multiplying by
		\(\langle\xi\rangle^{s_2-|s_1|}\), squaring, and summing over
		\(\xi\in\mathbb Z\), the conclusion follows from
		\eqref{eq:mixed-partial-characterization}.
	\end{proof}
	
	\begin{remark}
		The \(C^\infty\) assumption on \(a\), and hence on the associated
		phase \(A\), is imposed to avoid tracking the precise finite
		regularity assumptions required by the conjugation estimates.
		For each fixed Sobolev order \(s_1\), the proof of
		Proposition~\ref{prop:conjugation-mixed-smoothness} involves only
		finitely many derivatives of \(A\), and therefore only finitely many
		derivatives of \(a\). Corresponding finite-regularity versions can
		thus be formulated, although we do not pursue the optimal regularity
		assumptions here.
	\end{remark}

	%=================================================
	\subsection{Regularity estimates by conjugation} \
	%=================================================
	
	We now combine the normal-form reduction with the
	constant-coefficient result of Section~\ref{sec:constant-coefficients}.
	
	\begin{theorem}\label{theo:regularity-by-conjugation}
		Let \(a\in C^\infty(\mathbb T;\mathbb R)\), let
		\(k_1,k_2\in\mathbb R\), and suppose that
		\(u\in\mathcal D'(\mathbb T^2)\) is a solution of
		\(Lu=f\), with 
		\(f\in H_{\operatorname{mix}}^{(k_1,k_2)}(\mathbb T^2)\).
		
		Assume that \(a_0\in\mathbb R\setminus\mathbb Q\) has finite
		irrationality measure. Then, for every \(r>\mu(a_0)\) and every
		\(\theta\in[-r,1]\),
		\begin{equation*}\label{eq:regularity-by-conjugation}
			u\in
			H_{\operatorname{mix}}^{
				\left(
				k_1+\theta,\,
				k_2+1-r-\theta-|k_1|-|k_1+\theta|
				\right)}
			(\mathbb T^2).
		\end{equation*}
		
		Moreover,
		\begin{equation*}\label{eq:regularity-by-conjugation-estimate}
			\left\|
			u-\widehat u(0,0)
			\right\|_{
				H_{\operatorname{mix}}^{
					\left(
					k_1+\theta,\,
					k_2+1-r-\theta-|k_1|-|k_1+\theta|
					\right)}
				(\mathbb T^2)}
			\lesssim_{a,k_1,r,\theta}
			\|f\|_{
				H_{\operatorname{mix}}^{(k_1,k_2)}
				(\mathbb T^2)}.
		\end{equation*}
	\end{theorem}
	
	\begin{proof}
		Set
		\[
		g\coloneq\Psi_Af,
		\qquad
		w\coloneq\Psi_Au.
		\]
		
		By Proposition~\ref{prop:conjugation-mixed-smoothness},
		\(g\in H_{\operatorname{mix}}^{(k_1,k_2-|k_1|)}(\mathbb T^2)\), 
		and the intertwining identity gives
		\(L_0w=g\).
		
		Hence Theorem~\ref{theo_regularity_cte}\textup{(ii)} yields
		\[
		w\in
		H_{\operatorname{mix}}^{
			(k_1+\theta,\,
			k_2+1-r-\theta-|k_1|)}
		(\mathbb T^2).
		\]
		
		Since \(u=\Psi_{-A}w\), another application of
		Proposition~\ref{prop:conjugation-mixed-smoothness}, now at first
		order \(k_1+\theta\), gives
		\eqref{eq:regularity-by-conjugation}.
		
		For the quantitative estimate, the Jacobian determinant of the
		diffeomorphisms defining \(\Psi_{\pm A}\) is equal to \(1\), so
		\[
		\widehat{\Psi_{\pm A}v}(0,0)=\widehat v(0,0).
		\]
		
		Moreover, \(\Psi_{\pm A}\) preserve constants. Therefore
		\[
		u-\widehat u(0,0)
		=
		\Psi_{-A}
		\bigl(w-\widehat w(0,0)\bigr).
		\]
		
		Applying successively
		Proposition~\ref{prop:conjugation-mixed-smoothness},
		Theorem~\ref{theo_regularity_cte}\textup{(ii)}, and again
		Proposition~\ref{prop:conjugation-mixed-smoothness} gives
		\eqref{eq:regularity-by-conjugation-estimate}.
	\end{proof}
	
	\begin{remark}\label{rem:direct-versus-conjugation}
		The mixed smoothness estimates obtained by conjugation are contained
		in the admissible region furnished by the direct modal analysis of
		Theorem~\ref{theo_regularity_real}; in general, the inclusion is
		strict. For example, let \(k_1=0\) and \(0<\alpha\leq1\).
		Theorem~\ref{theo:regularity-by-conjugation} gives
		\[
		u\in
		H_{\operatorname{mix}}^{
			(\alpha,k_2+1-r-2\alpha)}(\mathbb T^2),
		\]
		whereas Theorem~\ref{theo_regularity_real}, applied with
		\(q=0\), gives
		\[
		u\in
		H_{\operatorname{mix}}^{
			(\alpha,k_2+1-r-\alpha)}(\mathbb T^2).
		\]
		Thus the conjugation may introduce an additional loss in the mixed
		smoothness scale. Its advantage appears instead in the isotropic
		Sobolev scale, where the conjugating operators preserve the Sobolev
		order.
	\end{remark}

	%=================================================
	\subsection{Isotropic regularity, existence, and classical consequences} \
	%=================================================
	
	The loss produced by the conjugation in the preceding subsection is
	specific to the mixed smoothness scale. In contrast, the normal-form
	conjugation preserves the order of isotropic Sobolev regularity.
	
	\begin{proposition}\label{prop:isotropic-conjugation}
		Let \(A\in C^\infty(\mathbb T;\mathbb R)\). For every
		\(s\in\mathbb R\), the operators \(\Psi_A\) and \(\Psi_{-A}\) are
		mutually inverse bounded automorphisms of \(H^s(\mathbb T^2)\).
	\end{proposition}
	
	\begin{proof}
		For \(m\in\mathbb N_0\), using the characterization of
		\(H^m(\mathbb T^2)\) in terms of the \(L^2\)-norms of derivatives,
		the chain rule, and the fact that the map \((t,x)\mapsto (t,x\pm A(t))\)
		preserves the Lebesgue measure on \(\mathbb T^2\), we obtain
		\[
		\|\Psi_{\pm A}u\|_{H^m(\mathbb T^2)}
		\lesssim_{A,m}
		\|u\|_{H^m(\mathbb T^2)}.
		\]
		
		Since the underlying diffeomorphisms have Jacobian determinant
		equal to \(1\), the operators are unitary on \(L^2(\mathbb T^2)\),
		with
		\[
		\Psi_A^*=\Psi_{-A}.
		\]
		
		Duality gives boundedness at negative integer Sobolev orders, and
		interpolation yields boundedness on \(H^s(\mathbb T^2)\) for every
		\(s\in\mathbb R\). Since \(\Psi_A\) and \(\Psi_{-A}\) are mutually
		inverse, they are bounded automorphisms of \(H^s(\mathbb T^2)\).
	\end{proof}
	
	\begin{corollary}\label{cor:isotropic-regularity-variable}
		Let \(s\in\mathbb R\), and suppose that \(f\in H^s(\mathbb T^2)\).
		Assume that \(a_0\in\mathbb R\setminus\mathbb Q\) has finite
		irrationality measure. Then every distributional solution
		\(u\in\mathcal D'(\mathbb T^2)\) of \(Lu=f\)  satisfies
		\[
		u\in H^{s+1-r}(\mathbb T^2)
		\]
		for every \(r>\mu(a_0)\).
		
		If, in addition, \(\widehat f(0,0)=0\), 
		then there exists a unique distributional solution \(u\) of
		\(Lu=f\) satisfying \(\widehat u(0,0)=0\). 
		This solution belongs to \(H^{s+1-r}(\mathbb T^2)\) for every
		\(r>\mu(a_0)\).
	\end{corollary}
	
	\begin{proof}
		By Proposition~\ref{prop:isotropic-conjugation}, 
		\(g\coloneq\Psi_Af\in H^s(\mathbb T^2)\). 		
		If \(u\) is a distributional solution of \(Lu=f\) and
		\(w\coloneq\Psi_Au\), then \(L_0w=g\).
		
		Proposition~\ref{prop:isotropic-regularity-constant} therefore gives
		\(w\in H^{s+1-r}(\mathbb T^2)\) 
		for every \(r>\mu(a_0)\). Applying \(\Psi_{-A}\) and using again
		Proposition~\ref{prop:isotropic-conjugation}, we obtain
		\(u\in H^{s+1-r}(\mathbb T^2)\).
		
		Assume now that \(\widehat f(0,0)=0\). Since the conjugation
		preserves the mean, \(\widehat g(0,0)=0\). 
		Moreover, Proposition~\ref{prop:isotropic-mixed-embeddings} gives
		\[
		H^s(\mathbb T^2)
		\hookrightarrow
		\begin{cases}
			H_{\operatorname{mix}}^{(0,s)}(\mathbb T^2),
			& s\geq0,\\
			H_{\operatorname{mix}}^{(s,s)}(\mathbb T^2),
			& s<0.
		\end{cases}
		\]
		
		Hence Corollary~\ref{coro_existence_cte}\textup{(ii)} gives the unique
		solution \(w\) of \(L_0w=g\) satisfying \(\widehat w(0,0)=0\).
		Setting \(u=\Psi_{-A}w\) gives the desired solution of \(Lu=f\).
		Existence, uniqueness, and the condition \(\widehat u(0,0)=0\) are
		preserved by the conjugation. 
	\end{proof}
	
	Using the isotropic invariance of the normal-form conjugation together
	with the sharpness results of \cite{KowKir2026-jfa}, we obtain the
	following variable-coefficient counterpart of the corresponding
	constant-coefficient result.
	
	\begin{corollary}\label{cor:sharpness-isotropic-variable}
		Let \(a\in C^\infty(\mathbb T;\mathbb R)\), and consider
		\(L=\partial_t-a(t)\partial_x\).
		Assume that \(a_0\) is irrational and
		\(\mu(a_0)<\infty\). Then, for every \(s\in\mathbb R\) and every
		\(\rho<\mu(a_0)\), there exists
		\[
		f\in H^s(\mathbb T^2)
		\]
		such that \(Lu=f\) admits a distributional solution, but no
		distributional solution belongs to
		\[
		H^{s+1-\rho}(\mathbb T^2).
		\]
	\end{corollary}
	
	\begin{proof}
		By the constant-coefficient sharpness result of
		\cite{KowKir2026-jfa}, there exists \(g\in H^s(\mathbb T^2)\)
		such that \(L_0w=g\) is distributionally solvable, but no solution
		belongs to \(H^{s+1-\rho}(\mathbb T^2)\).
		
		Set \(f=\Psi_{-A}g\). Then \(f\in H^s(\mathbb T^2)\), and every
		solution \(w\) of \(L_0w=g\) gives a solution
		\(u=\Psi_{-A}w\) of \(Lu=f\). If some solution \(v\) of \(Lv=f\)
		belonged to \(H^{s+1-\rho}(\mathbb T^2)\), then
		\(\Psi_Av\) would belong to the same Sobolev space and satisfy
		\(L_0(\Psi_Av)=g\), a contradiction.
	\end{proof}
	
	\begin{corollary}\label{cor:isotropic-variable-classical-regularity}
		Assume that
		\(a_0\in\mathbb R\setminus\mathbb Q\) has finite irrationality
		measure. Let \(k>\mu(a_0)\), and suppose that
		\(u\in\mathcal D'(\mathbb T^2)\) is a solution of \(Lu=f\), with
		\(f\in H^k(\mathbb T^2)\). Then
		\[
		u\in C^m(\mathbb T^2),
		\qquad
		m\coloneq
		\left\lceil k-\mu(a_0)\right\rceil-1.
		\]
		
		In particular, if \(k\in\mathbb N_0\), the same conclusion holds
		for every \(f\in C^k(\mathbb T^2)\).
		
		If, in addition, \(\widehat f(0,0)=0\), then there exists a unique
		distributional solution satisfying \(\widehat u(0,0)=0\), and this
		solution has the regularity stated above.
	\end{corollary}
	
	\begin{proof}
		By the definition of \(m\) we have \(m<k-\mu(a_0)\). 
		Choose \(r\) such that
		\[
		\mu(a_0)<r<k-m.
		\]
		
		Corollary~\ref{cor:isotropic-regularity-variable} gives
		\(u\in H^{k+1-r}(\mathbb T^2)\). 
		Since \(k+1-r>m+1\), the Sobolev embedding on \(\mathbb T^2\)
		yields \(u\in C^m(\mathbb T^2)\). The existence and uniqueness
		statement follows from the corresponding part of
		Corollary~\ref{cor:isotropic-regularity-variable}.
	\end{proof}
	
	\begin{remark}
		Set \(\varepsilon\coloneq k-\mu(a_0)>0\). For isotropic data
		\(f\in H^k(\mathbb T^2)\), the direct mixed smoothness argument of
		Section~\ref{sec:variable-real} gives
		\[
		u\in
		C^{
			\left\lceil(\varepsilon+\frac12)/3\right\rceil-1
		}(\mathbb T^2),
		\]
		whereas the normal-form argument gives
		\[
		u\in C^{\lceil\varepsilon\rceil-1}(\mathbb T^2).
		\]
		
		The two conclusions agree when \(0<\varepsilon\leq1\), while the
		normal-form result is strictly stronger when \(\varepsilon>1\).
		The mixed smoothness estimates nevertheless retain separate
		information on the regularity in the two variables.
	\end{remark}
	
	\begin{example}\label{example:isotropic-quadratic-variable}
		Consider
		\[
		L
		=
		\partial_t
		-
		\frac{2}{1+\cos^2t}\partial_x.
		\]
		Its averaged coefficient is \(a_0=\sqrt2\). Since \(\sqrt2\) is a
		quadratic irrational, the endpoint Diophantine estimate holds with
		\(r=\mu(\sqrt2)=2\).
		
		Hence, if \(f\in H^s(\mathbb T^2)\) and \(Lu=f\), then
		\(u\in H^{s-1}(\mathbb T^2)\).
		In particular, for \(k>2\), 
		\[
		f\in H^k(\mathbb T^2)
		\quad\Longrightarrow\quad
		u\in C^{k-3}(\mathbb T^2),
		\qquad k\in\mathbb N.
		\]
		
		For example, if \(k=6\), the direct mixed smoothness method of
		Section~\ref{sec:variable-real} guarantees \(u\in C^1(\mathbb T^2)\),
		whereas the normal-form argument gives \(u\in C^3(\mathbb T^2)\).
		
		If, in addition, \(\widehat f(0,0)=0\), there exists a unique
		distributional solution satisfying \(\widehat u(0,0)=0\).
		More generally, any two distributional solutions differ by a
		constant.
	\end{example}
	
	\begin{remark}
		One may also consider
		\[
		L
		=
		\partial_t
		-
		2\arccos\left(\frac12\cos t\right)\partial_x.
		\]
		
		Since \(a(t+\pi)=2\pi-a(t)\), the averaged coefficient is \(a_0=\pi\). 
		In contrast with the preceding example, the exact value of 
		\(\mu(\pi)\) is not known. The corresponding regularity thresholds are 
		therefore naturally expressed in terms of \(\mu(\pi)\). Known upper 
		bounds for \(\mu(\pi)\) nevertheless yield explicit, though possibly
		nonoptimal, Sobolev regularity estimates.
	\end{remark} 
	
	%=================================================
	\subsection{Transferred rational and Liouville obstructions} \
	%=================================================
	
	Given a map
	\(\boldsymbol{\rho}\colon\mathbb R^2\longrightarrow\mathbb R^2\),
	we say that the associated family of target orders
	\[
	\left\{
	\mathbf{k}-\boldsymbol{\rho}(\mathbf{k})
	:
	\mathbf{k}\in\mathbb R^2
	\right\}
	\]
	is cofinal if, for every \(s\geq0\), there exists
	\(\mathbf{k}=(k_1,k_2)\in\mathbb R^2\) such that
	\begin{equation*}\label{eq:cofinal-orders-normal-form}
		k_j-\rho_j(\mathbf{k})\geq s,
		\qquad
		j=1,2.
	\end{equation*}
	
	Equivalently, both components of the target orders can be made
	arbitrarily large simultaneously, or, in other words, the map
	\[
	\mathbf{k}
	\longmapsto
	\min_{j=1,2}
	\bigl\{k_j-\rho_j(\mathbf{k})\bigr\}
	\]
	is unbounded above.
	
	The next proposition shows that no universal family of mixed
	smoothness estimates with this property can hold when \(a_0\) is
	rational or Liouville.
	
	\begin{proposition}\label{prop:normal-form-arithmetic-obstructions}
		Assume that \(a_0\) is rational or Liouville, and let
		\(\boldsymbol{\rho}\colon\mathbb R^2\to\mathbb R^2\) be such that
		the associated family of target orders is cofinal. Then there exist
		\[
		\mathbf{k}\in\mathbb R^2,\qquad
		f\in
		H_{\operatorname{mix}}^{\mathbf{k}}(\mathbb T^2),
		\qquad
		u\in\mathcal D'(\mathbb T^2),
		\]
		such that \(Lu=f\) and
		\[
		u\notin
		H_{\operatorname{mix}}^{
			\mathbf{k}-\boldsymbol{\rho}(\mathbf{k})}(\mathbb T^2).
		\]
	\end{proposition}
	
	\begin{proof}
		By Propositions~\ref{prop:rational-no-mixed-regularity} and
		\ref{prop:liouville-no-mixed-regularity}, according to the
		arithmetic case, there exist
		\[
		w\in\mathcal D'(\mathbb T^2)\setminus L^2(\mathbb T^2),
		\qquad
		g\in C^\infty(\mathbb T^2),
		\qquad
		L_0w=g.
		\]
		
		Set
		\[
		u\coloneq\Psi_{-A}w,
		\qquad
		f\coloneq\Psi_{-A}g.
		\]
		
		Then \(Lu=f\) and \(f\in C^\infty(\mathbb T^2)\). Moreover,
		\(u\notin C^\infty(\mathbb T^2)\), since otherwise
		\(w=\Psi_Au\) would be smooth. Hence there exists
		\(s\in\mathbb N_0\) such that \(u\notin H^s(\mathbb T^2)\).
		
		By cofinality, choose \(\mathbf{k}\in\mathbb R^2\) such that
		\[
		k_j-\rho_j(\mathbf{k})\geq s,
		\qquad j=1,2.
		\]
		
		Since \(f\) is smooth,
		\(f\in
		H_{\operatorname{mix}}^{\mathbf{k}}(\mathbb T^2)\).
		If
		\[
		u\in
		H_{\operatorname{mix}}^{
			\mathbf{k}-\boldsymbol{\rho}(\mathbf{k})}(\mathbb T^2),
		\]
		then monotonicity and
		Proposition~\ref{prop:isotropic-mixed-embeddings} would give
		\[
		u\in
		H_{\operatorname{mix}}^{(s,s)}(\mathbb T^2)
		\hookrightarrow H^s(\mathbb T^2),
		\]
		a contradiction.
	\end{proof}
	
	\begin{remark}
		Proposition~\ref{prop:normal-form-arithmetic-obstructions} does not
		assert the absence of mixed smoothness estimates in the rational or
		Liouville cases; it shows that no universal cofinal family of target
		orders can hold.
		
		In contrast, suppose that
		\(a_0\in\mathbb R\setminus\mathbb Q\) has finite irrationality
		measure, fix \(r>\mu(a_0)\), and set
		\[
		\boldsymbol{\rho}_r(k_1,k_2)
		\coloneq
		\bigl(0,r-1+|k_1|\bigr).
		\]
		
		By Theorem~\ref{theo_regularity_real}, with
		\(q=k_1\) and \(\theta=0\),
		\[
		f\in
		H_{\operatorname{mix}}^{(k_1,k_2)}(\mathbb T^2)
		\quad\Longrightarrow\quad
		u\in
		H_{\operatorname{mix}}^{
			(k_1,k_2)-\boldsymbol{\rho}_r(k_1,k_2)}(\mathbb T^2).
		\]
		
		This family is cofinal: given \(s\geq0\), the choice
		\((k_1,k_2)=(s,2s+r-1)\)  gives
		\[
		(k_1,k_2)-\boldsymbol{\rho}_r(k_1,k_2)=(s,s).
		\]
		
		Thus finite irrationality measure permits an explicit cofinal
		family of mixed smoothness estimates, whereas rational and
		Liouville averages obstruct every universal cofinal family of this
		type.
	\end{remark}

	%=====================================================
	%=====================================================	
	\section*{Acknowledgments} 
	%=====================================================
	%=====================================================

	The first and second authors were supported in part by CNPq -- Brasil (grants 301573/2025-5, 302048/2026-0, and 302867/2026-0). The third author was supported, in part, by the S\~ao Paulo Research Foundation (FAPESP), Brazil, Process
	Numbers 2025/08151-5 and 2025/26296-0.

\end{document}